\documentclass[12pt]{amsart}
  
\usepackage{amsfonts, enumitem,amsmath,amsthm,amssymb,mathrsfs}

\numberwithin{equation}{section}

\providecommand{\binom}[2]{{#1\choose#2}}

\newcommand{\codim}{\operatorname{codim}}
\renewcommand{\geq}{\geqslant}
\renewcommand{\leq}{\leqslant}
\newcommand{\Osh}{{\mathcal O}}                        
\renewcommand{\H}{\mathrm{H}}                          
\newcommand{\Sym}{\operatorname{Sym}}
\newcommand{\K}{\mathrm{K}}                            

\newcommand{\Ish}{\mathcal{I}}
\newcommand{\G}{\mathbb{G}}

\newcommand{\GL}{\operatorname{GL}}

\newcommand{\Vol}{\operatorname{Vol}}
\newcommand{\spec}{\operatorname{Spec}}
\newcommand{\ord}{\mathrm{ord}}

\newcommand{\KK}{\mathbf{K}}
\newcommand{\FF}{\mathbf{F}}
\renewcommand{\AA}{\mathbb{A}} 
\newcommand{\NN}{\mathbb{N}} 
\newcommand{\PP}{\mathbb{P}} 
\newcommand{\QQ}{\mathbb{Q}} 
\newcommand{\RR}{\mathbb{R}} 
\newcommand{\ZZ}{\mathbb{Z}} 

\newtheorem{theorem}{Theorem}[section]

\newtheorem{corollary}[theorem]{Corollary}
\newtheorem{proposition}[theorem]{Proposition}

\theoremstyle{definition}
\newtheorem{defn}[theorem]{Definition}
\newtheorem{remark}[theorem]{Remark}
\newtheorem{example}[theorem]{Example}

\begin{document}
\title[Expectations, Chow weights and Roth's theorem]{Expectations, Concave Transforms, Chow weights and Roth's theorem for varieties}

\author{Nathan Grieve}
\address{Department of Mathematics \& Computer Science,
Royal Military College of Canada, P.O. Box 17000,
Station Forces, Kingston, ON, K7K 7B4, Canada
}
\address{School of Mathematics and Statistics, 4302 Herzberg Laboratories, Carleton University, 1125 Colonel By Drive, Ottawa, ON, K1S 5B6, Canada
}
\address{
D\'{e}partement de math\'{e}matiques, Universit\'{e} du Qu\'{e}bec \`a Montr\'{e}al, Local PK-5151, 201 Avenue du Pr\'{e}sident-Kennedy, Montr\'{e}al, QC, H2X 3Y7, Canada}
\email{nathan.m.grieve@gmail.com}%

\begin{abstract} 
We present a new perspective, which is at the intersection of K-stability and Diophantine arithmetic geometry.  It builds on work of Boucksom and Chen \cite{Boucksom:Chen:2011}, Boucksom-et-al \cite{Boucksom:Kuronya:Maclean:Szemberg:2015}, \cite{Boucksom-Hisamoto-Jonsson:2016}, Codogni and  Patakfalvi \cite{Codogni:Patakfalvi:2021}, Fujita \cite{Fujita:2018}, \cite{Fujita:2019}, Li \cite{Li:2017a}, Ferretti \cite{Ferretti:2000}, \cite{Ferretti:2003}, McKinnon and Roth \cite{McKinnon-Roth}, Ru and Vojta \cite{Ru:Vojta:2016}, Ru and Wang \cite{Ru:Wang:2016}, \cite{Ru:Wang:2021}, Heier and Levin \cite{Heier:Levin:2017} and our related works (including \cite{Grieve:Function:Fields}, \cite{Grieve:2018:autissier}, \cite{Grieve:Divisorial:Instab:Vojta}, \cite{Grieve:toric:gcd:2019}, \cite{Grieve:points:bounded:degree} and \cite{Grieve:MVT:2019}).  Our main results may be described via the expectations of the Duistermaat-Heckman measures.  For example, we prove that  the expected orders of vanishing for very ample linear series along subvarieties may be calculated via the theory of Chow weights and Okounkov bodies.  This result provides a novel unified viewpoint to the existing works that surround these topics. It expands on work of Mumford \cite{Mumford:77}, Odaka \cite{Odaka:2013} and Fujita \cite{Fujita:2018}. Moreover, it has applications for Diophantine approximation and K-stability.  As one result of this flavour, we define a concept of uniform arithmetic $\K$-stability for Kawamata log terminal pairs.  We then establish a form of K.~F.~ Roth's celebrated approximation theorem for certain collections of divisorial valuations which fail to be uniformly arithmetically $\K$-stable.  This result builds on the Roth type theorems of McKinnon and Roth \cite{McKinnon-Roth}.  Moreover, it complements the concept of arithmetic canonical boundedness, in the sense of McKinnon and Santriano \cite{McKinnon:Santriano:2021}, in addition to our results that establish instances of Vojta's Main Conjecture for $\QQ$-Fano varieties \cite{Grieve:Divisorial:Instab:Vojta}.
\end{abstract}

\thanks{
I thank the Natural Sciences and Engineering Research Council of Canada for their support through my grants DGECR-2021-00218 and RGPIN-2021-03821. }  

\maketitle

\section{Introduction}\label{intro}

The purpose of this article is to present a new unified viewpoint, which is at the intersection of higher dimensional Diophantine arithmetic geometry, positivity for line bundles on projective varieties and concepts that surround $\K$-stability.  As one example,  we expect that complexity of rational points is measured on rational curves \cite[Conjecture 2.7]{McKinnon:2007}.  Here, as follows from our results, this complexity can be interpreted via the theories of Chow forms and Okounkov bodies.

The Diophantine arithmetic part has been of considerable recent interest (\cite{McKinnon-Roth}, \cite{Ru:Wang:2016}, \cite{Grieve:Function:Fields}, \cite{Ru:Vojta:2016}, \cite{Ru:Vojta:2021}, \cite{Grieve:2018:autissier}, \cite{Grieve:toric:gcd:2019}, \cite{Grieve:Divisorial:Instab:Vojta}, \cite{Grieve:points:bounded:degree}, \cite{Heier:Levin:2017}, \cite{He:Ru:2022}, \cite{Grieve:HN:polygons:Laws:Large:Numbers}).  It has origins in the work of Faltings and W\"{u}stholz \cite{Faltings:Wustholz} and Ferretti \cite{Ferretti:2000}, \cite{Ferretti:2003}.  The formulation of these results in a manner that extends the classical statement of K. F. Roth's theorem is due to McKinnon-Roth \cite{McKinnon-Roth} and Ru-Wang \cite{Ru:Wang:2016}.  That these results could be deduced, as an application of Schmidt's Subspace Theorem, is one theme of \cite{Grieve:Function:Fields} and \cite{Ru:Wang:2016}.  That K-instability for $\QQ$-Fano varieties implies instances of Vojta's Main Conjecture was established in \cite{Grieve:Divisorial:Instab:Vojta}, building on \cite{Ru:Vojta:2016} and \cite{Grieve:2018:autissier}.  Here, we pursue further this circle of ideas (see Theorem \ref{Roth:constant:theorem}).

The more recent convex geometric part, expanding on earlier work of Donaldson \cite{Donaldson:2002}, is made possible by work of Boucksom and Chen \cite{Boucksom:Chen:2011} and Boucksom-et-al \cite{Boucksom:Kuronya:Maclean:Szemberg:2015}, \cite{Boucksom-Hisamoto-Jonsson:2016}.    The manner in which these works are related to the Roth type theorems from \cite{McKinnon-Roth} and the Arithmetic General Theorem from \cite{Ru:Vojta:2016} is discussed in \cite{Grieve:2018:autissier}, \cite{Grieve:toric:gcd:2019} and \cite{Grieve:MVT:2019}.    In particular, \cite[Proposition 6.1]{Grieve:toric:gcd:2019} allows for a more precise form of \cite[Corollary 1.12]{Ru:Vojta:2016}.  Among other results, the calculations with polytopes, as a measure of expected orders of vanishing along torus invariant subvarieties, from \cite{Grieve:toric:gcd:2019}, are special instances of Theorem \ref{volume:constant:expectation:theorem}.  

To formulate our results, let $\overline{\KK}$ be an algebraically closed characteristic zero field.  For the arithmetic applications that we have in mind, including Theorems \ref{Roth:constant:theorem} and \ref{Roth:Constants:Thm}, $\overline{\KK}$ will be a fixed algebraic closure of a number field $\KK$.  

Let $X$ be an irreducible normal projective variety over $\overline{\KK}$ and 
$$Z \subsetneq X$$ 
a proper subscheme.  Fix an ample line bundle $L$ on $X$ and put 
\begin{equation}\label{asym:vol:constant}
\beta_Z(L) := \int_0^{\infty} \frac{\operatorname{Vol}(\pi^* L - t E)}{\operatorname{Vol}(L)} \mathrm{d}t \text{.}
\end{equation}
This is the \emph{expected order of vanishing} of $X$ along $Z$ with respect to $L$.  Here
 $E$ is the exceptional divisor of the blowing-up
 $$\pi \colon X' = \mathrm{Bl}_Z(X) \rightarrow X$$ 
 of $X$ along $Z$.
 
The quantity $\beta_Z(L)$, defined in \eqref{asym:vol:constant}, has origins in a number of different works.  Within the arithmetic context, for example, it is known to have appeared in unpublished work of P. Salberger and was popularized by McKinnon and Roth \cite{McKinnon-Roth}.   Its relation to the main arithmetic invariant of \cite{Ru:Vojta:2016}, building on earlier work of Autissier \cite{Autissier:2011}, was established in \cite{Grieve:2018:autissier}.  Within the context of $\K$-stability for $(X,L)$, there is an interpretation as the expectation of the Duistermaat-Heckman measure \cite{Boucksom-Hisamoto-Jonsson:2016}, \cite{Li:2017a}, \cite{Grieve:MVT:2019}.

A special case of Theorem \ref{volume:constant:expectation:theorem} reads in the following way.

\begin{theorem}\label{volume:constant:expectation:theorem:intro}  Let $L$ be a very ample line bundle, on a normal projective variety $X$, and let 
$$Z \subsetneq X$$ 
be a proper subscheme.  In this context, the expected order of vanishing $\beta_Z(L)$ can be described as a normalized Chow weight
\begin{equation}\label{beta:expectation:3:intro}
\beta_Z(L) = \frac{e_X(\mathbf{c})}{(\dim X + 1) (\operatorname{deg}_L X)}.
\end{equation}
\end{theorem}
 
In \eqref{beta:expectation:3:intro},  as in \cite[Corollary 3.2]{Grieve:2018:autissier}, $e_X(\mathbf{c})$ is the \emph{Chow weight} of $X$ with respect to a linearly normal embedding 
$$X \hookrightarrow \PP^n_{\overline{\KK}}$$ 
that is induced by an \emph{inflectionary basis}; i.e., a basis for $\H^0(X,L)$, which is compatible with the filtration that is given by orders of vanishing along $E$.  Here, the \emph{weight vector} 
$$\mathbf{c} = (a_0(L),\dots,a_n(L))$$ 
is determined by the vanishing numbers of $L$ with respect to the filtration that is induced by the exceptional divisor $E$.  It is natural to refer to such embeddings as \emph{inflectionary embeddings}.  We refer to Section \ref{discrete:measures:chow} for more details.

Theorem \ref{volume:constant:expectation:theorem:intro} is a consequence of Theorem \ref{Linearly:normal:Chow:Weights:Thm} and Theorem \ref{volume:constant:expectation:theorem}.  Theorem \ref{Linearly:normal:Chow:Weights:Thm} is of an independent interest and refines work of Mumford, including \cite[Section 2]{Mumford:77}, Fujita \cite[Section 4]{Fujita:2018} and Boucksom-et-al \cite[Section 2.6]{Boucksom-Hisamoto-Jonsson:2016}.

On the other hand, Theorem \ref{volume:constant:expectation:theorem} contains a number of equivalent descriptions of $\beta_Z(L)$.  It shows that it can be expressed as an expectation which is determined by the concave transform of the Okounkov body of $L$.  

As in \cite[Proposition 3.1]{Grieve:2018:autissier}, this aesthetically pleasing result establishes its equivalence with the main arithmetic invariant of \cite{Ru:Vojta:2016}, which builds on \cite[Definition 2.4]{Autissier:2011}.  Moreover, Theorems \ref{volume:constant:expectation:theorem:intro} and \ref{volume:constant:expectation:theorem} allow for a conceptual interpretation of the main invariants of K-stability, including those that have origins in the work of Donaldson \cite{Donaldson:2002}.  A slight extension to the works \cite{Boucksom:Chen:2011} and \cite{Boucksom:Kuronya:Maclean:Szemberg:2015}, which is exposed in \cite{Grieve:MVT:2019}, is one aspect to our proof of these results.

We also mention that the quantity $\beta_Z(L)$ is related to the \emph{Seshadri constant}
$$
\epsilon(L;Z) := \sup \{t \in \RR_{\geq 0} : \pi^*L - tE \text{ is nef} \} 
$$
\cite[page 206]{Ross:Thomas:2007}.
For instance, when $X$ is nonsingular, it follows from \cite[Theorem 4.2]{Heier:Levin:2017}, building on the well-known inequality \cite[Corollary 4.2]{McKinnon-Roth}, combined with the equivalence of these definitions as was established in \cite{Grieve:2018:autissier}, that 
\begin{equation}\label{epsilon:beta:eqn}
\beta_Z(L) \geq \frac{\codim_X(Z)}{\dim X + 1} \epsilon(L;Z) \text{.}
\end{equation}
In Example \ref{codim:seshadri:lower:bound}, we expand upon the approach of \cite[Section 4]{Zhu:2020} to obtain a similar estimate, for the case that $Z$ is a complete intersection in the complete linear system $|L|$.  This estimate involves the incomplete beta functions. 

When $X$ is a $\QQ$-Fano variety and $L = \K_X^{\otimes -1}$, the inequality \eqref{epsilon:beta:eqn} can be used to give an upper bound for the Donaldson-Futaki invariant of the basic test configurations determined by dreamy subvarieties of $X$.  This is explained in Section \ref{concave:transform:examples} (see Examples \ref{DF:Example} and \ref{codim:seshadri:lower:bound}).

In \cite{Grieve:Divisorial:Instab:Vojta}, it is observed that the valuative criteria for $\K$-instability of $\QQ$-Fano varieties has arithmetic consequences for Vojta's Main Conjecture.  Recall, that such valuative criteria, for $\K$-stability of $\QQ$-Fano varieties (see \cite{Fujita:2019}) involve the quantities $\beta_Z(L)$ for the case that $L = \K_X^{\otimes -1}$.  An extension to this circle of ideas has recently been given in \cite{Dervan:Legendre:2022}.

The starting point for the arithmetic observations of \cite{Grieve:Divisorial:Instab:Vojta}, include \cite[Theorem 10.1]{McKinnon-Roth}, \cite[Corollary 1.3]{Grieve:2018:autissier} and \cite[Corollary 1.12]{Ru:Vojta:2016}.  Within the context of toric varieties, more detailed calculations that illustrate calculation of the quantity $\beta_Z(L)$, via Okounkov bodies, may be found in  \cite{Grieve:toric:gcd:2019}.  They should be seen as special cases of Theorem \ref{volume:constant:expectation:theorem}, which we establish here.  

As some additional results, which are of an independent interest, in Section \ref{very:ample:test:config}, we explore the manner in which weight vectors pertain to very ample test configurations for linearly normal embedded projective varieties.  These results (see Theorems \ref{weight:vectors:very:ample:test:configs} and  \ref{projective:normal:test:configs:Rees:algebras}), in addition to illustrating the results from Section \ref{discrete:measures:chow}, provide refinements, for the case of linearly normal embeddings, to the well known result of Ross and Thomas \cite[Proposition 3.7]{Ross:Thomas:2007}.

Finally, we obtain arithmetic consequences of our results.
For example, in the direction of Vojta's Main Conjecture and its relation to $\K$-instability, for polarized projective varieties, we establish Theorem \ref{Roth:constant:theorem}.  It is a form of Roth's celebrated approximation theorem and complements \cite[Theorem 10.1]{McKinnon-Roth} and \cite[Theorem 1.1 and Corollary 1.2]{Grieve:Divisorial:Instab:Vojta}.

In order to formulate Theorem \ref{Roth:constant:theorem}, consider the case of a Kawamata log terminal pair $(X,\Delta)$ defined over a number field $\KK$.  Let $L$ be an ample line bundle on $X$.  The idea is to consider a form of \emph{arithmetic uniform $\K$-stability}.  In more precise terms, fix a finite set of places $S$ and for each $v \in S$, let $E_v$ be a prime divisor over $X$, with field of definition some finite extension $\FF$ of $\KK$ with the property that $\KK \subseteq \FF \subseteq \overline{\KK}$.  Here, $\overline{\KK}$ is a fixed algebraic closure of $\KK$.  We  refer to Sections \ref{real:valuations:discrete:measures} and \ref{approx:constants} for more details in regards to our conventions about divisorial valuations and prime divisors over $X$.

Motivated by the concept of \emph{uniform stability} for Kawamata log terminal Fano pairs, see \cite[Definition 4.8]{Codogni:Patakfalvi:2021}, here, we say that $(X,\Delta)$ is \emph{not arithmetically $\K$-stable} with respect to $L$ and $E_v$, for $v \in S$, if
\begin{equation}\label{arithmetic:K:stable:inequality}
1 < \sum_{v \in S} A(E_v,X,\Delta) < \sum_{v \in S} \beta_{E_v}(L)R_v
\end{equation}
for some positive constants $R_v$.  In \eqref{arithmetic:K:stable:inequality}, we say that such positive constants $R_v$, for $v \in S$, are \emph{arithmetically $\K$-destabilizing Roth constants}.  Moreover, $A(E_v,X,\Delta)$ denotes the \emph{log discrepancy} of $E_v$ with respect to $(X,\Delta)$.

\begin{theorem}\label{Roth:constant:theorem}
Let $(X,\Delta)$ be a Kawamata log terminal pair defined over a number field $\KK$.  Let $L$ be an ample line bundle on $X$ and defined over $\KK$.  Fix a finite set of places $S \subseteq M_{\KK}$.  For each $v \in S$, let $E_v$ be a prime divisor  over $X$ and having field of definition some finite extension field of $\KK$.  Assume that $(X,\Delta)$ is not arithmetically $\K$-stable with respect to $L$ and $E_v$, for each 
$v \in S\text{.}$ 
Moreover, suppose that
$R_v \in \RR_{>0}\text{,}$ 
for $v \in S$, are destabilizing Roth constants as in \eqref{arithmetic:K:stable:inequality}.  Then, there exists a proper Zariski closed subset 
$$W \subsetneq X\text{,}$$ 
defined over $\KK$, and at least one place $v \in S$, so that the approximation constants $\alpha_{E_v}(\{x_i\},L)$, which are defined  for all infinite sequences of distinct $\KK$-rational points 
$$\{x_i\} \subseteq X(\KK) \setminus W(\KK)\text{,}$$ 
satisfy the inequality
$$
\alpha_{E_v}(\{x_i\},L) \geq 1 / R_v \text{.}
$$  
\end{theorem}

We prove Theorem \ref{Roth:constant:theorem} in Section \ref{Roth:constants:delta:invariant}.  It is a consequence of Theorem \ref{Roth:Constants:Thm}, which is a result that is of an independent interest.  Specifically, Theorem \ref{Roth:Constants:Thm} uses Schmidt's Subspace Theorem to establish a logarithmic form of \cite[Theorem 5.1]{McKinnon-Roth}.  The approximation constants $\alpha_{E_v}(\{x_i\},L)$, which arise in Theorem \ref{Roth:constant:theorem}, are defined in Definition \ref{divisorial:val:approx:contant}.

\subsection*{Acknowledgements}  
This work began while I was a postdoctoral fellow at the University of New Brunswick, Fredericton NB, where I was financially supported by an AARMS postdoctoral fellowship.  
It benefited from a visit to the Atlantic Algebra Centre, St.~John's NL, during March of 2017.  Some later portions of this work were completed while I was a postdoctoral fellow at Michigan State University.  They profited from visits to CIRGET, Montreal, during May of 2019, ICERM, Providence, during June of 2019, NCTS and Institute of Mathematics Academia Sinica, Taipei, during June of 2019, and the American Institute of Mathematics, San Jose, during January of 2020.  
Finally, I thank the Natural Sciences and Engineering Research Council of Canada for their support through my grants DGECR-2021-00218 and RGPIN-2021-03821 and colleagues and anonymous referees for their interest and helpful thoughtful comments.

\section{Preliminaries}

In this section, we collect various facts about growth of linear series and filtrations of section rings.  

\subsection*{Conventions and terminology} Here, and elsewhere, all varieties are assumed to be irreducible and reduced.  Unless mentioned otherwise, all varieties are defined over $\overline{\KK}$, an algebraically closed characteristic zero field.  In Sections \ref{approx:constants} and \ref{Roth:constants:delta:invariant}, $\overline{\KK}$ will be a fixed algebraic closure of a number field $\KK$ and we will consider varieties over $\KK$.  Finally, similar to \cite[Definition 2.24]{Kollar:Mori:1998}, by a \emph{prime divisor over} a projective variety $X$, is meant a nonzero, irreducible, reduced and effective Cartier divisor that is supported on some normal proper model of $X$.  

\subsection*{Asymptotic growth of linear series}
Recall positivity for projective varieties, especially \cite[Section 2.2.C]{Laz}, \cite{Ein-et-al:2005}
and \cite{Ein-et-al-09}.  Let $X$ be a $d$-dimensional projective variety.  Denote the section ring of a line bundle $L$ on $X$ as
$$R(L) = R(X,L) := \bigoplus_{m \geq 0} \H^0(X,L^{\otimes m}) = \bigoplus_{m \geq 0} R_m \text{,}$$  
and the \emph{volume} of $L$ by
$$ 
\Vol_X(L) = \Vol(L) := \limsup_{m \to \infty} \frac{h^0(X,L^{\otimes m})}{m^d / d!} = \lim_{m \to \infty} \frac{h^0(X,L^{\otimes m})}{m^d / d!} \text{.}
$$
Recall that the volume function $\Vol(\cdot)$ extends to give a continuous function on the real N\'{e}ron-Severi space of $X$.

Next, let 
$$Z \subsetneq X$$ 
be a subvariety of dimension $\ell > 0$.  Then the \emph{restricted volume} of $L$ along $Z$ is
$$
\Vol_{X | Z}(L) := \limsup_{m\to \infty} \frac{h^0(X|Z, L^{\otimes m})}{m^\ell / \ell!} = \lim_{m \to \infty } \frac{h^0(X|Z, L^{\otimes m})}{m^\ell / \ell!}.
$$
Here $h^0(X|Z, L^{\otimes m})$ denotes the dimension of the image of the restriction map 
$$
\H^0(X,L^{\otimes m}) \rightarrow \H^0(Z, L^{\otimes m}|_Z).
$$
There is a concept of restricted volume of numerical classes of $\RR$-divisors along $Z$.  For this, it is required that $Z$ is not contained in the augmented base locus of the given $\RR$-divisor.

Given a function $f(m)$, represented by a polynomial of degree at most $k$ in $m$ for large $m$, we denote by $\mathrm{n.l.c.}(f)$ the \emph{normalized leading coefficient} of $f$.  This is defined to be the number $e$ for which
$ f(m) = e m^k / k! + \mathrm{O}(m^{k-1})$.

\subsection*{Filtered linear series}  Fix a \emph{big} line bundle 
$$L = \Osh_X(D)$$ 
on $X$, a $d$-dimensional projective variety.  All of the \emph{filtrations} 
$$\mathcal{F}^\bullet = \mathcal{F}^\bullet R(L)\text{,}$$ 
that we consider here are assumed to be \emph{decreasing}, \emph{left-continuous} 
and \emph{multiplicative}.   These conventions are similar to those of \cite[Definition 1.3]{Boucksom:Chen:2011}, \cite[Section 1.1]{Boucksom-Hisamoto-Jonsson:2016},  \cite[Definition 4.1]{Fujita:2018} and \cite[Section 2]{Grieve:MVT:2019}.  

In particular, by our conventions,  such filtrations $\mathcal{F}^\bullet$ have the properties that 
\begin{enumerate}
\item{
if $m \in \ZZ_{\geq 0}$, then
\begin{itemize}
\item{$\mathcal{F}^t R_m \subseteq \mathcal{F}^{t'} R_m$, for all $t,t'\in \RR$ with $t \geq t'$;}
\item{$\mathcal{F}^t R_m = \mathcal{F}^{t - \epsilon} R_m$, for all $t \in \RR$ and all  sufficiently small $\epsilon > 0$; and}
\end{itemize}
}
\item{ 
if $m,m' \in \ZZ_{\geq 0}$ and $t,t'\in\RR$, then
$$ \mathcal{F}^t R_m \cdot \mathcal{F}^{t'}R_{m'} \subseteq \mathcal{F}^{t + t'} R_{m+m'}\text{.}$$ 
}
\end{enumerate}
Moreover, if 
$$\mathcal{F}^t R_m = \mathcal{F}^{\lceil t \rceil} R_m\text{,}$$ 
for all $t \in \RR$, then $\mathcal{F}^\bullet R$ is said to be a \emph{$\ZZ$-filtration}.

The \emph{vanishing numbers} of a filtration $\mathcal{F}^\bullet$ have the property that
$$
a_{\min}(L^{\otimes m}) = a_0(L^{\otimes m}) \leq \dots \leq a_{n_m}(L^{\otimes m}) = a_{\max}(L^{\otimes m}) \text{;}
$$ 
they are defined by the condition that
$$
a_j( L^{\otimes m}) = a_j( L^{\otimes m},\mathcal{F}^\bullet) := \inf \left\{ t \in \RR : \codim \mathcal{F}^t \H^0(X, L^{\otimes m}) \geq j + 1 \right\}.
$$
Put
$$a_{\min}(||L||) := \liminf_{m\to \infty} \frac{a_{\min}( L^{\otimes m})}{ m } 
$$
and
$$
a_{\max}(||L||) := \limsup_{m\to \infty}  \frac{a_{\max}( L^{\otimes m}) } { m }.
$$

In what follows, our interest are those $\RR$-filtrations $\mathcal{F}^\bullet R$ which are 
 \emph{pointwise bounded below} and \emph{linearly bounded above}.
 Especially, they have the property that the quantities $a_{\min}(||L||)$ and $a_{\min}(||L||)$, as defined above, are well-defined real numbers.
 
In more precise terms, $\mathcal{F}^\bullet R$ is \emph{pointwise bounded below} if for all $m\geq 0$, there exists a real number $T \in \RR$ with
$$
\mathcal{F}^t R_m = R_m \text{,}
$$
for all $t \leq  T$. 
On the other hand, if there exists a constant $C > 0$, which has the property that 
$$a_{\min}(L^{\otimes m}) \geq - C m \text{,}$$
for all 
$m \geq 0$,  
then $\mathcal{F}^\bullet$ is said to be \emph{linearly bounded below}.  If 
$$
a_{\max}(L^{\otimes m}) \leq Cm \text{,}
$$
for all 
$m \geq 0$ 
and some constant 
$C>0$, 
then $\mathcal{F}^\bullet R$ is said to be \emph{linearly bounded above}.  

If $\mathcal{F}^\bullet R$ is both linearly bounded above and linearly bounded below, then $\mathcal{F}^\bullet R$ is said to be \emph{linearly bounded}, or simply \emph{bounded}.  Further, when no confusion is likely, all linearly bounded above filtrations are assumed to be pointwise bounded below.   For such filtrations we tacitly assume that $a_{\min }(L^{\otimes m}) \geq 0$, for all $m \geq 0$.

Henceforth, given a linearly bounded above and pointwise bounded below $\RR$-filtration $\mathcal{F}^\bullet R$, for each $a \in \RR$, let $\delta_a(t)$ denote the Dirac distribution with support $a \in \RR$.   Moreover, for each integer $m  \geq 0$, define discrete measures on the real line $\RR$, by the condition that
\begin{equation}\label{discrete:measure:nu}
\nu_m = \nu_m(t) := \frac{1}{h^0(X, L^{\otimes m})} \sum_j \delta_{m^{-1} a_j( L^{\otimes m})}(t),
\end{equation}
 \cite[page 1213]{Boucksom:Chen:2011} and \cite[page 813]{Boucksom:Kuronya:Maclean:Szemberg:2015}.  As emphasized in   \cite{Boucksom:Chen:2011}, \cite{Boucksom:Kuronya:Maclean:Szemberg:2015} and  \cite{Grieve:MVT:2019}, we can consider the limit of the discrete measures $\nu_m$
$$ \nu = \lim_{m \to \infty} \nu_m.$$

We conclude this section by describing two important examples of filtrations (see Example \ref{Filtrations:Example}).  In Section \ref{very:ample:test:config}, we discuss in detail the filtrations that arise via linearly normal embeddings and one parameter subgroups of the general linear group.

\begin{example}\label{Filtrations:Example}
Let $L$ be a big line bundle on a projective variety $X$ with section ring 
$$
R(X,L) := \bigoplus_{m \geq 0} R_m = \bigoplus_{m \geq 0}  \H^0(X,L^{\otimes m}) \text{.}
$$
\begin{enumerate}
\item{When $X$ is normal, an important class of filtrations is obtained via orders of vanishing along subvarieties.  Such filtrations may be constructed in the following way.  First, if 
$Z \subsetneq X$ 
is a proper subscheme, then let 
$$\pi \colon X' \rightarrow X$$ 
be the (normalized) blowing-up morphism, with exceptional divisor $E$, of $X$ along $Z$.  Then $X'$ is the normalization of 
$$\operatorname{Bl}_Z(X) \rightarrow X\text{,}$$ 
the blowing-up of $X$ along $Z$.  
The corresponding filtration $\mathcal{F}^\bullet R$, then has the property that 
$$
\mathcal{F}^t R_m = \H^0\left(X, L^{\otimes m} \otimes \pi_* (\Osh_{X'}(- \lceil t E \rceil ))\right) = \H^0\left(X', \pi^* L^{\otimes m} \otimes \Osh_{X'}(- \lceil t E \rceil)\right)
$$
for all 
$m \geq 0$ 
and all 
$t \in \ZZ \text{.}$  
It is easily verified, see \cite[Example 2.2]{Grieve:MVT:2019}, that such filtrations are, indeed, bounded.
}\label{orders:vanishing}
\item{Assume that $L$ is ample.  If $\mathcal{F}^\bullet R$ is an $\RR$-filtration of $R$, then for each 
$m \in \ZZ_{\geq 0}$ 
and each 
$t \in \ZZ\text{,}$ 
let
$I_{(m,t)}$ denote the image of the evaluation homomorphism
$$
\mathcal{F}^t R_m \otimes L^{\otimes - m} \rightarrow \Osh_X \text{.}
$$
Moreover, set 
$$
\bar{\mathcal{F}}^t R_m := \H^0(X,L^{\otimes m} \cdot I_{(m,t)}) \text{.}
$$
Then, it is known by \cite[Section 4.3]{Fujita:2018}, for instance, that $\bar{\mathcal{F}}^\bullet R_m$ is an $\RR$-filtration in the sense described above.  This is called the \emph{saturation} of $\mathcal{F}^\bullet R$.  In case that 
\begin{equation}\label{saturated:defn}
\bar{\mathcal{F}}^t R_m = \mathcal{F}^t R_m \text{,}
\end{equation} 
for all 
$t \in \ZZ$ 
and all 
$m \in \ZZ_{\geq 0}\text{,}$ 
then the filtration $\mathcal{F}^\bullet R$ is called \emph{saturated}.  In case that \eqref{saturated:defn} is valid for all sufficiently large $m > 0$ and all $t \in \RR$, then, here, we say that $\mathcal{F}^\bullet$ is  \emph{asymptotically saturated}.  For example, when $X$ is normal, the filtrations that are obtained via orders of vanishing along subvarieties, as described in (a) above, are known to be saturated.  Indeed, this follows from \cite[Lemma 1.8]{Boucksom-Hisamoto-Jonsson:2016} for example.
}\label{saturated:example}
\end{enumerate}
\end{example}

\section{Convex bodies and concave transforms}
Let $X$ be a projective variety of dimension $d$.
The theory of \emph{Okounkov bodies}, \cite{Lazarsfeld:Mustata:2009}, \cite{Kaveh:Khovanskii:2012}, associates a convex body $\Delta(L)$ to each big line bundle $L$ on $X$.  More generally, there is a concept of Okounkov body $\Delta(W_\bullet)$, for 
$$W_\bullet = \bigoplus_{m \geq 0} W_m$$ 
a \emph{graded linear series} 
$$W_m \subseteq \H^0(X,L^{\otimes m})\text{.}$$
Recall that such convex bodies are determined by the valuation like functions
$$
\nu_{X_\bullet} = \nu_{X_\bullet,L^{\otimes m}} \colon \H^0(X,L^{\otimes m}) \rightarrow \ZZ^d \bigcup \{\infty\}
$$
which depend on the choice of a fixed \emph{admissible flag} of irreducible codimension $i$ subvarieties, for $i = 1,\dots,d$.  We refer to \cite[Section 3]{Grieve:MVT:2019} for a more detailed overview.

Given a pointwise bounded below and linearly bounded above $\RR$-filtration 
$$\mathcal{F}^\bullet = \mathcal{F}^\bullet R(L)\text{,}$$ 
put, for each $t \in \RR$,
$$
W^t_m := \mathcal{F}^{mt} \H^0(X,L^{\otimes m}) 
$$
and 
$$
W_\bullet^t := \bigoplus_{m \geq 0} W^t_m  \text{.}
$$  
The \emph{concave transform} of $\mathcal{F}^\bullet$, \cite[Definition 1.8]{Boucksom:Chen:2011}, is the concave function
$$
G_{\mathcal{F}^\bullet} \colon \Delta(L) \rightarrow [ -\infty,\infty [
$$
which is defined by the conditions that
$$
G_{\mathcal{F}^\bullet}(y) := \sup \{ t \in \RR : y \in \Delta(W^t_\bullet) \} \text{.}
$$

For later use we record Proposition \ref{expectation:concave:transform} below.  It is a consequence of \cite[Theorem 1.1]{Grieve:MVT:2019}.

\begin{proposition}\label{expectation:concave:transform} Suppose that $L$ is a big line bundle on a $d$-dimensional projective variety $X$.  Let $\mathcal{F}^\bullet$ be a  pointwise bounded below and linearly bounded above $\RR$-filtration of its section ring and let
 $\lambda$ be the restriction of Lebesgue measure of $\RR^d$ to $\Delta(L)^{\circ}$, the interior of the Okounkov body of $L$, and let $(G_{\mathcal{F}^\bullet})_*\lambda$ be its pushforward to $\RR$ by the concave transform $G_{\mathcal{F}^\bullet}$.  Then the limit expectation, $\mathbb{E}(\nu)$, of $\nu$ can be described as
\begin{equation}\label{expectation:concave:transform:eqn}
\mathbb{E}(\nu) 
= \frac{d!}{\operatorname{Vol}(L)} \int_0^{a_{\max}(||L||)} t \cdot \mathrm{d}((G_{\mathcal{F}^\bullet})_*\lambda).
\end{equation}
\end{proposition}

\begin{proof} 
If
$$ 
g(t) := \lim_{m\to \infty} m^{-d} \dim \mathcal{F}^{mt} \H^0(X, L^{\otimes m}),
$$
then it is a consequence of \cite[Proof of Theorem 1.1]{Grieve:MVT:2019}, see also \cite[Proof of Theorem 1.11]{Boucksom:Chen:2011}, that the expectation $\mathbb{E}(\nu)$, of the limit measure $\nu$, can be described as
\begin{equation}\label{mu:expectation:integral:volume:function}
\mathbb{E}(\nu) = \frac{d!}{\operatorname{Vol}(L)} \int_0^{a_{\max}(||L||)} g(t) \mathrm{d}t = \frac{d!}{\operatorname{Vol}(L)}  
 \int_0^{a_{\max}(||L||)} t \cdot \mathrm{d}((G_{\mathcal{F}^\bullet})_*\lambda).
\end{equation}
\end{proof}

\section{Linearly normal embeddings, filtrations and Chow weights}\label{discrete:measures:chow}

The goal of this section is to establish Theorem \ref{Linearly:normal:Chow:Weights:Thm}.  It is essentially due to Mumford (see \cite[Section 2]{Mumford:77} and \cite[Section 3]{Morrison:1980}).  Here, our approach is based on the refinements which have arisen within the area  of K-stability.  They include work of Odaka,  \cite{Odaka:2013}, Boucksom, Hisamoto and Jonsson, \cite{Boucksom-Hisamoto-Jonsson:2016}, and Fujita \cite{Fujita:2018}.  

Theorem \ref{Linearly:normal:Chow:Weights:Thm} treats the case of linearly normal embeddings and \emph{asymptotically saturated} $\ZZ$-filtrations of the section ring. As in \cite[Corollary 3.2]{Grieve:2018:autissier}, it includes the class of filtrations which arise from \emph{inflectionary embeddings}, i.e.,  those embeddings which are defined by a basis which is compatible with the filtration of the section ring that is induced by orders of vanishing along a given subvariety.  

Note also, that Theorem \ref{Linearly:normal:Chow:Weights:Thm} is an improved form, for the case of very ample linear series, of \cite[Lemma 4.7]{Fujita:2018}.  To provide further motivation to these results, in Section \ref{very:ample:test:config}, we consider the case of those filtrations that arise from ample test configurations for linearly normal projective varieties.

In order to formulate Proposition \ref{vanishing:number:chow:weights:prop}, let $L$ be a big line bundle on a projective variety $X$ and let 
$$\mathcal{F}^\bullet = \mathcal{F}^\bullet R$$ 
be a pointwise bounded below and linearly bounded above $\RR$-filtration of its section ring.   
$$R = R(X,L) = \bigoplus_{m \geq 0} \H^0(X,L^{\otimes m}) = \bigoplus_{m \geq 0} R_m \text{.}$$  
Recall, that our conventions are such that
$a_{\min}(L^{\otimes m}) \geq 0\text{,}$
for all $m \geq 0$.  In particular, the vanishing numbers 
$$
a_j(L^{\otimes m}) = \inf \left\{ t \in \RR : \operatorname{codim} \mathcal{F}^t \H^0(X,L^{\otimes m}) \geq j + 1 \right\} \text{,}
$$
for 
$m \in \ZZ_{\geq 0}\text{,}$ 
have the property that
$$
0 \leq a_{\min}(L^{\otimes m}) = a_0(L^{\otimes m}) \leq \dots \leq a_{n_m}(L^{\otimes m}) = a_{\max}(L^{\otimes m}) \text{.}
$$

If $m \in \ZZ_{\geq 0}$, then let
\begin{equation}\label{Hilbert:weights:defn:0}
 s(m,\mathcal{F}^\bullet) := \sum_{j=0}^{n_m} a_j(L^{\otimes m})
\end{equation}
We say that this is the \emph{$m$th Hilbert weight} for $L$ with respect to the pointwise bounded below and linearly bounded above filtration $\mathcal{F}^\bullet$.  

Note that, in general, the Hilbert weight $s(m,\mathcal{F}^{\bullet})$ is not determined by the \emph{weight vector} of $\H^0(X,L)$.  But, compare with the case of very ample line bundles and filtrations via orders of vanishing along a subvariety, and more generally the case of very ample line bundles and bounded asymptotically saturated filtrations (see for instance Theorem \ref{Linearly:normal:Chow:Weights:Thm} and Corollary \ref{hilbert:weight:asymptotic:expansion}).

For later use, we record the following proposition.  It applies, in particular, to inflectionary embeddings in the sense described above.

\begin{proposition}\label{vanishing:number:chow:weights:prop}      
In the above setting just described, namely if $L$ is a big line bundle on a projective variety $X$ and if $\mathcal{F}^\bullet$ is a pointwise bounded below and linearly bounded above $\RR$-filtration of the section ring $R(X,L)$, then 
\begin{equation}\label{prop:4.1:eqn3}
\mathbb{E}(\nu_m) = \frac{s(m,\mathcal{F}^\bullet)}{m h^0(X, L^{\otimes m})}.
\end{equation}
\end{proposition}
\begin{proof}
Recall that by \eqref{Hilbert:weights:defn:0}
\begin{equation}\label{prop:4.1:eqn1}
s(m,\mathcal{F}^\bullet) = \sum_{j=0}^{n_m} a_j(L^{\otimes m})
\end{equation}
and it also follows from the description of the measures $\nu_m$, given in \eqref{discrete:measure:nu}, that
\begin{equation}\label{prop:4.1:eqn2}
h^0(X,L^{\otimes m}) \mathbb{E}(\nu_m) = \sum_{j=0}^{n_m} \int_0^{a_{\max}(||L||)} t \cdot \delta_{a_j(L^{\otimes m})m^{-1}}(t) \mathrm{d}t = \sum_{j=0}^{n_m} a_j(L^{\otimes m}) m^{-1}\text{.}
\end{equation}
Together \eqref{prop:4.1:eqn1} and \eqref{prop:4.1:eqn2} yield \eqref{prop:4.1:eqn3}.
\end{proof}

We now proceed to discuss Theorem \ref{Linearly:normal:Chow:Weights:Thm}.  It builds on earlier work of Mumford, including \cite{Mumford:77} and \cite{Morrison:1980}, Odaka, for instance  \cite{Odaka:2013}, and Fujita, especially \cite[Section 4]{Fujita:2018}.  
The conclusion of Theorem \ref{Linearly:normal:Chow:Weights:Thm} applies to the class of filtrations, which are obtained via weight vectors as described in Section \ref{very:ample:test:config}.  

It also applies to \emph{inflectionary embeddings}, by which we mean those embeddings of normal varieties, which are compatible with a filtration that is induced by orders of vanishing along a given subvariety.  We refer to Example \ref{Filtrations:Example} (a), for a more detailed description of such filtrations.  As one additional remark,  as will become apparent in our proof of Theorem \ref{Linearly:normal:Chow:Weights:Thm}, our conventions for \emph{Hilbert weights} differ, by a minus sign, from those of \cite[page 417]{Fujita:2018}.

In order to formulate Theorem \ref{Linearly:normal:Chow:Weights:Thm}, let $L$ be a very ample line bundle on a projective variety $X$ and let $\mathcal{F}^\bullet R$ be an asymptotically saturated bounded $\ZZ$-filtration of its section ring.   Assume that 
$a_{\min}(L) = 0\text{.}$

Given 
$m \in \ZZ_{\geq 0}$ 
and 
$t \in \RR\text{,}$ 
let 
$$
I^{\mathcal{F}^\bullet}_{(m,t)} = I_{(m,t)}
$$
be the image of the natural evaluation homomorphism
$$
\mathcal{F}^t R_m \otimes L^{\otimes -m} \rightarrow \Osh_X \text{;}
$$
put
$$
\bar{\mathcal{F}}^t R_m = \H^0(X,L^{\otimes m} \cdot I_{(m,t)}) \text{.}
$$
 Let
$$
X_{\AA^1} := X \times \AA^1_{\overline{\KK}} 
$$
and denote by $L_{\AA^1}$ the pullback of $L$ to $X_{\AA^1}$, under the projection map to the first factor.
 Then
 \begin{equation}\label{flag:ideal:global:sections:eqn:1}
 \H^0(X_{\AA^1},L_{\AA^1}) = \H^0(X,L) \otimes \overline{\KK}[x] \text{.}
 \end{equation}
 
 Let
$$\Ish \subseteq \Osh_{X_{\AA^1}}$$ 
be the ideal sheaf that is defined by the condition that
$$
\Ish = I_{(1,a_{\max}(L))} + I_{(1,a_{\max}(L) - 1)} z + \dots + I_{(1,1)} z^{a_{\max}(L) -1} + (z^{a_{\max}(L)}) \text{.}
$$
For all 
$m \geq 0\text{,}$ write
$$
\Ish^m \subseteq \Osh_{X_{\AA^1}}
$$
as
$$
\Ish^m = J_{(m,m a_{\max}(L)) } + J_{(m, ma_{\max}(L) -1)}z^1 + \dots + J_{(m,1)}z^{m a_{\max}(L) - 1 }+ (z^{m a_{\max}(L)}) \text{,}
$$
where $J_{(m,t)}$ is the ideal sheaf of $\Osh_X$ that is defined by the condition that 
$$
J_{(m,t)} : = \sum_{\substack{ t_1+\dots+t_m = t \\ \text{ and } \\ t_1,\dots,t_m \in [0,a_{\max}(L)] \bigcap \ZZ } } I_{(1,t_1)} \cdot \hdots \cdot I_{(1,t_m)} \text{.}
$$

Then
$$
L_{\AA^1}^{\otimes m} \cdot \Ish^m   =  \sum_{i=0}^{m a_{\max}(L)}  L^{\otimes m} \cdot J_{(m, ma_{\max}(L) -i)}z^i 
$$
and 
\begin{equation}\label{flag:ideal:global:sections:eqn:2}
\H^0(X_{\AA^1}, L_{\AA^1}^{\otimes m} \cdot \Ish^m) = \sum_{i=0}^{m a_{\max}(L)}  \H^0(X,L^{\otimes m} \cdot J_{(m, ma_{\max}(L) -i)}) z^i \text{.}
\end{equation}

With this above notation and hypothesis, Theorem \ref{Linearly:normal:Chow:Weights:Thm} is stated in the following way.

\begin{theorem}\label{Linearly:normal:Chow:Weights:Thm}
Let $L$ be a very ample line bundle on a projective variety $X$ and let $\mathcal{F}^\bullet R$ be an asymptotically saturated bounded $\ZZ$-filtration of its section ring.   Assume that 
$a_{\min}(L) = 0\text{.}$
Then, with the notations and hypothesis as above, for all $t \in \RR$ and all sufficiently large $m > 0$, it holds true that
$$
\H^0(X,L^{\otimes m} \cdot J_{(m,t)}) = \bar{\mathcal{F}}^t R_m = \H^0(X,L^{\otimes m} \cdot I_{(m,t)}) \text{.}
$$
Furthermore
\begin{align*}
s(m,\mathbf{c}) & := \sum_{i=0}^{n_m} a_i(L^{\otimes m}) \\
& = m a_{\max}(L) h^0(X,L^{\otimes m}) - \sum_{t = 1}^{a_{\max}(L^{\otimes m})} h^0\left(X,L^{\otimes m} \cdot J_{(m,t)}\right) \\
&  =  \dim_{\overline{\KK} } \H^0(X_{\AA^1},  L^{\otimes m}_{\AA^1}) / \H^0(X_{\AA^1},  L^{\otimes m}_{\AA^1} \cdot \Ish^m) 
\text{.}
\end{align*}
\end{theorem}

\begin{proof}  The proof of Theorem \ref{Linearly:normal:Chow:Weights:Thm} is achieved via the theory of flag ideals (compare, for instance, with \cite[Section 3]{Morrison:1980}, \cite[Definition 3.1]{Odaka:2013} and \cite[Section 4]{Fujita:2018}).

First of all, by assumption, $a_{\min}(L) = 0$.  Thus,
the vanishing numbers $a_j(L^{\otimes m})$, for $j = 0,\dots,n_m$ and all sufficiently large 
$m > 0$, 
have the property that
$$
0 = a_{\min}(L^{\otimes m}) = a_0(L^{\otimes m}) \leq \dots \leq a_{n_m}(L^{\otimes m}) = a_{\max}(L^{\otimes m}) \text{.}
$$

Next, let $V_{(m,t)}$ be the image of the natural map
$$
\bigoplus_{\substack{ t_1 + \dots + t_m = t \\
\text{ and } \\
t_1,\dots,t_m \in [0, a_{\max}(L)] \bigcap \ZZ }
} \bar{\mathcal{F}}^{t_1} R_1 \otimes \dots \otimes \bar{\mathcal{F}}^{t_m} R_1 \rightarrow \bar{\mathcal{F}}^t R_{m } \text{.}
$$
Then, the ideal sheaves 
$$
J_{(m,t)} : = \sum_{\substack{ t_1+\dots+t_m = t \\ \text{ and } \\ t_1,\dots,t_m \in [0,a_{\max}(L)] \bigcap \ZZ } } I_{(1,t_1)} \cdot \hdots \cdot I_{(1,t_m)} 
$$
are the image of the natural evaluation map
$$
V_{(m,t)} \otimes L^{\otimes -m} \rightarrow \Osh_X \text{.}
$$
Moreover, observe that if 
$$t \in [0, m a_{\max}(L)] \text{,}$$ 
then
\begin{equation}\label{flag:ideal:eqn}
V_{(m,t)} \subseteq \H^0\left(X,L^{\otimes m} \cdot J_{(m,t)}\right) \subseteq \bar{\mathcal{F}}^t R_{m } \text{.}
\end{equation}

Now, by assumption, the polarizing line bundle, $L$ is assumed to be very ample.  In particular, 
the natural map
\begin{equation}\label{sym:alg:surj}
\Sym^m \left(R_1\right) \rightarrow R_m
\end{equation}
is surjective, for all sufficiently large $m > 0$.
Moreover, this map is compatible with the filtration $\mathcal{F}^\bullet R$ and the induced filtration $\mathcal{F}^\bullet\left(\Sym^\bullet\left(R_1\right)\right)$.  
In particular, for all $t \in \RR$  
$$
\mathcal{F}^t \left(\Sym^m\left(R_1\right)\right) \subseteq \mathcal{F}^t R_m \text{,}
$$
for all $m \in \ZZ_{\geq 0}$.

Further, recall that, by assumption, the filtration $\mathcal{F}^\bullet R$ is asymptotically saturated. Thus, for all sufficiently large $m > 0$, 
$$
\mathcal{F}^t R_m = \bar{\mathcal{F}}^t R_m := \H^0(X,L^{\otimes m} \cdot I_{(m,t)}) \text{,}
$$
for all 
$t \in \ZZ$.  It is then evident that, because of the asymptotic surjectivity of \eqref{sym:alg:surj}, that for all sufficiently large $m>0$, the natural maps
$$
\bigoplus_{\substack{ t_1 + \dots + t_{m} = t \\
t_1,\dots,t_m \in [0,a_{\max}(L)] \bigcap \ZZ  }} \bar{\mathcal{F}}^{t_1} R_1 \otimes \dots \otimes \bar{\mathcal{F}}^{t_m} R_1 \rightarrow \bar{\mathcal{F}}^t R_m 
$$
are surjective for all  
$t \in \ZZ$.

Moreover, since the maps \eqref{sym:alg:surj} are surjective, for all sufficiently large $m > 0$, it follows that
$$
m a_{\max}(L) = a_{\max}(L^{\otimes m})
$$
and
$$
V_{(m,t)} = \bar{\mathcal{F}}^t R_{m} \text{.}
$$
Thus, by \eqref{flag:ideal:eqn}, for all sufficiently large $m > 0$, it follows that 
$$
 \H^0(X,L^{\otimes m} \cdot J_{(m,t)}) = \H^0(X, L^{\otimes m} \cdot I_{(m,t)}) \text{.}
$$

In conclusion, the above discussion implies, similar to \cite[Proposition 3.2]{Morrison:1980}, that for all sufficiently large $m > 0$, it then holds true that
\begin{align*}
s(m,\mathbf{c}) & := \sum_{i=0}^{n_m} a_i(L^{\otimes m}) \\
& = m a_{\max}(L) h^0(X,L^{\otimes m}) - \sum_{t = 1}^{a_{\max}(L^{\otimes m})} h^0(X,L^{\otimes m} \cdot J_{(m,t)}) \\
& = 
\dim_{\overline{\KK} } \H^0(X_{\AA^1},  L^{\otimes m}_{\AA^1})/ \H^0(X_{\AA^1},  L^{\otimes m}_{\AA^1} \cdot \Ish^m) 
\text{.}
\end{align*}
Here, the last most inequality follows from \eqref{flag:ideal:global:sections:eqn:1} and \eqref{flag:ideal:global:sections:eqn:2}.
\end{proof}

In the setting of Theorem \ref{Linearly:normal:Chow:Weights:Thm}, so, in particular, $L$ is a very ample line bundle on a projective variety $X$ and $\mathcal{F}^\bullet R$ is an asymptotically saturated bounded $\ZZ$-filtration of its section ring, we say that
$$
s(m,\mathbf{c}) = 
s(m,\mathcal{F}^\bullet) := \sum_{j=0}^{n_m} a_j(L^{\otimes m})
$$
is the \emph{$m$th Hilbert weight} for $L$ with respect to the filtration $\mathcal{F}^\bullet$.  We also say that this is the \emph{$m$th Hilbert weight} for $L$ with respect to the \emph{weight vector} 
\begin{equation}\label{filtration:weight:vector}
\mathbf{c} = (a_0(L),\dots,a_n(L)) \in \ZZ^{n+1} \text{.}
\end{equation}

As one consequence, Theorem \ref{Linearly:normal:Chow:Weights:Thm} allows for a concept of \emph{normalized Chow weight} for $(X,L)$ with respect to the filtration $\mathcal{F}^\bullet$.  Indeed, this follows from Corollary \ref{hilbert:weight:asymptotic:expansion}.

\begin{corollary}\label{hilbert:weight:asymptotic:expansion}
With the notations and hypothesis of Theorem \ref{Linearly:normal:Chow:Weights:Thm}, for all sufficiently large $m > 0$, the Hilbert weight 
$$s(m,\mathbf{c}) := \sum_{i = 0}^{n_m} a_i(L^{\otimes m})$$ 
may be expressed as a polynomial
$$
s(m,\mathbf{c}) = \frac{e_X(\mathbf{c})}{(d+1)!} m^{d+1} + \mathrm{O}(m^d) \text{,}
$$
for a suitable constant $e_X(\mathbf{c})$.  Here, $d = \dim X$.
\end{corollary}

\begin{proof} 
By Theorem \ref{Linearly:normal:Chow:Weights:Thm}, if $m > 0$ is sufficiently large, then
\begin{equation}\label{hilbert:weight:polynomial}
s(m,\mathbf{c}) = \dim_{\overline{\KK} } \H^0(X_{\AA^1},  L^{\otimes m}_{\AA^1})/ \H^0(X_{\AA^1}, L^{\otimes m}_{\AA^1} \cdot \Ish^m) 
\text{.}
\end{equation}

The conclusion that is desired by Corollary \ref{hilbert:weight:asymptotic:expansion} then follows, upon noting that, by \cite[Proposition 2.6]{Mumford:77}, the right hand side of \eqref{hilbert:weight:polynomial},  for all sufficiently large $m > 0$, is a degree $d+1$ polynomial in $m$, for $d = \dim X$.
\end{proof}

Continuing with the setting of Corollary \ref{hilbert:weight:asymptotic:expansion}, consider $e_X(\mathbf{c})$, the \emph{normalized leading coefficient} of $s(m,\mathbf{c})$.  Then 
\begin{equation}\label{leading:coeff:eqn}
e_X(\mathbf{c}) = \operatorname{n.l.c.}(s(m,\mathbf{c})) \text{.}
\end{equation}
In what follows, we say that 
\begin{equation}\label{normalized:chow:weight:defn}
\frac{e_X(\mathbf{c})}{(d+1)(\operatorname{deg}_L X)}
\end{equation}
 is the \emph{normalized Chow weight for $(X,L)$, with respect to the weight vector \eqref{filtration:weight:vector}}.

It is possible to describe the normalized Chow weight \eqref{normalized:chow:weight:defn} in terms of the \emph{shifted weight vector}
$$
\mathbf{r} = (r_0,\dots,r_n) \in \ZZ^{n+1} \text{,}
$$ 
where
$$
r_i = r_i(L) = a_{\max}(L) - a_i(L) \text{,}
$$
for $i = 0,\dots, n$.  In more precise terms, put
\begin{equation}\label{chow:eqn:shifted:vector:eqn:0}
e_X(\mathbf{r}) = (d+1)(\operatorname{deg}_L X)  a_{\max}(L) - e_X(\mathbf{c}) \text{.}
\end{equation}
Then
\begin{equation}\label{chow:eqn:shifted:weight:vector:eqn}
\frac{e_X(\mathbf{c})}{(d+1)(\operatorname{deg}_L X) } = a_{\max}(L) - \frac{e_X(\mathbf{r})}{(d+1)(\operatorname{deg}_L X)} \text{.}
\end{equation}

Also, by combining \eqref{leading:coeff:eqn} and \eqref{chow:eqn:shifted:vector:eqn:0}, it follows that
\begin{align}\label{chow:eqn:shifted:weight:vector:eqn:prime}
s(m,\mathbf{c}) & = \frac{-e_X(\mathbf{r}) + (d+1)(\operatorname{deg}_L X)a_n(L)}{(d+1)!} m^{d+1} - \mathrm{O}(m^d)  \nonumber \\
& = \frac{- e_X(\mathbf{r}) }{(d+1)! } m^{d+1} + m \operatorname{Hilb}_{S/I}(m) a_n(L) + \mathrm{O}(m^d) \text{,}
\end{align}
for $\operatorname{Hilb}_{S/I}(m)$ the Hilbert polynomial of $X$ in $\PP^n$ with respect to the given embedding.

Theorem \ref{vanishing:number:normalized:chow:weights:cor} establishes the fact that the normalized Chow weight \eqref{normalized:chow:weight:defn} admits a description in terms of the Duistermaat-Heckman measure and the theory of Okounkov bodies.  It may be compared with the similar, although in general different, concept of Chow weights associated to test configurations (as arises in the K-stability literature including \cite{Szek:2015}). 

\begin{theorem}\label{vanishing:number:normalized:chow:weights:cor}
Let $L$ be a very ample line bundle on a $d$-dimensional projective variety $X$.  Let $\Delta(L)$ be its Okounkov body with respect to some admissible flag of subvarieties.  Let 
$$\mathcal{F}^\bullet = \mathcal{F}^\bullet R$$ 
be a saturated and bounded $\ZZ$-filtration of the section ring $R(X,L)$ and having  the property that $a_{\min}(L) = 0$.  Let
$$G_{\mathcal{F}^\bullet} \colon \Delta(L) \rightarrow [-\infty, \infty[ $$ 
be the concave transform of $\mathcal{F}^\bullet$.  Then, in this setting, the normalized Chow weight of $(X,L)$, with respect to $\mathcal{F}^\bullet$, is expressed, via $G_{\mathcal{F}^\bullet}$, in terms of the expectation of the Duistermaat-Heckman measure by the relation that
\begin{equation}\label{DH:measure:Chow:weight:eqn}
\frac{e_X(\mathbf{c})}{(d + 1)(\operatorname{deg}_L X)} = \frac{d!}{\Vol(L)} \int_0^{a_{\max}(||L||)} t \cdot \mathrm{d}\left(\left( G_{\mathcal{F}^\bullet}\right)_* \lambda \right) \text{.}
\end{equation}
\end{theorem}

\begin{proof}
Recall, that by Proposition \ref{vanishing:number:chow:weights:prop}
$$
\mathbb{E}(\nu_m) = \frac{s(m,\mathbf{c})}{m h^0(X,L^{\otimes m})} \text{.}
$$
Further, as noted in \eqref{leading:coeff:eqn}, Corollary \ref{hilbert:weight:asymptotic:expansion} implies that
$$
e_X(\mathbf{c}) = \operatorname{n.l.c.}(s(m,\mathbf{c})) \text{.}
$$

On the other hand, it is proven in \cite[Theorem 1.1]{Grieve:MVT:2019}, that
$$
\mathbb{E}(\nu) = \lim_{m \to \infty} \mathbb{E}(\nu_m)\text{.} 
$$
Thus, it follows in light of \eqref{chow:eqn:shifted:weight:vector:eqn} and \eqref{chow:eqn:shifted:weight:vector:eqn:prime}, that
\begin{equation}\label{chow:weight:proof:eqn}
\mathbb{E}(\nu) = \lim_{m \to \infty} \frac{s(m,\mathbf{c})}{m h^0(X,L^{\otimes m})} = \frac{e_X(\mathbf{c})}{(d+1)(\operatorname{deg}_L X)} \text{.}
\end{equation}
Finally, combining \eqref{chow:weight:proof:eqn} with Proposition \ref{expectation:concave:transform}, the conclusion is then that
$$
\frac{e_X(\mathbf{c})}{(d+1)(\operatorname{deg}_L X)} = \frac{d!}{\Vol(L)} \int_0^{a_{\max}(||L||)} t \cdot \mathrm{d}\left( \left( G_{\mathcal{F}^\bullet} \right)_*\lambda \right) \text{.}
$$
This establishes the desired equation \eqref{DH:measure:Chow:weight:eqn}.
\end{proof}

\section{One parameter subgroups, weight vectors, linearly normal embeddings and Rees algebras}\label{very:ample:test:config}

In this section, we study the correspondence between weight vectors and the ample test configurations that they determine, for linearly normal embeddings of projective varieties into projective space.  In doing so, we expose, in a slightly more refined form, the well known result of Ross and Thomas \cite[Proposition 3.7]{Ross:Thomas:2007}.

\subsection*{Linearly normal embeddings, one parameter subgroups and test configurations}
In what follows, we express the data of a one parameter subgroup, 
of the maximal torus, inside of the general linear group $\operatorname{GL}_{n+1}(\overline{\KK})$ in the form
\begin{equation}\label{1:PS:eqn1}
\rho(z) := z^{-k} \cdot \operatorname{diag}(z^{c_0},\dots,z^{c_n})
\end{equation}
for suitable integers
\begin{equation}\label{1:PS:eqn2}
k,c_0,\dots,c_n \in \ZZ
\end{equation}
with the property that
\begin{equation}\label{1:PS:eqn3}
0 \leq c_0 \leq \dots \leq c_n \text{.}
\end{equation}
We say that
\begin{equation}\label{1:PS:eqn4}
\mathbf{c} = (c_0,\dots,c_n) \in \ZZ^{n+1}
\end{equation}
is the one parameter subgroup's \emph{weight vector}.  Similarly, by a \emph{weight vector}, is meant a vector as in \eqref{1:PS:eqn4}, which has the property \eqref{1:PS:eqn3}.

We wish to clarify the manner in which weight vectors correspond to very ample test configurations for a given linearly normal projective variety.  But first, as in \cite[page 762]{Boucksom-Hisamoto-Jonsson:2016}, if $L$ is an ample line bundle on a projective variety $X$, then an \emph{ample test configuration} for $(X,L)$ consists of a flat and proper morphism 
$$\pi \colon \mathcal{X} \rightarrow \AA^1 \text{,}$$ 
together with a $\G_m$-action on $\mathcal{X}$, which lifts the canonical $\G_m$-action on $\AA^1$,  a $\G_m$-linearized $\pi$-ample line bundle $\mathcal{L}$ on $\mathcal{X}$ and an isomorphism 
$$(X,L) \simeq (\mathcal{X}_{z_0},\mathcal{L}|_{\mathcal{X}_{z_0}}) \text{,}$$ 
for one and hence all 
$$z_0 \in \AA^1 \setminus \{0\}\text{.}$$  
Denote the data of such a test configuration by $(\mathcal{X},\mathcal{L})$.  There is an evident concept of \emph{very ample test configuration} for linearly normal embeddings.  This is the context of Theorem \ref{weight:vectors:very:ample:test:configs} below.  It should be compared with {\cite[Proposition 3.7]{Ross:Thomas:2007}} and {\cite[Section 2.3]{Boucksom-Hisamoto-Jonsson:2016}}.

\begin{theorem}\label{weight:vectors:very:ample:test:configs}
Let $L$ be a very ample line bundle on a projective variety $X$.  Let $$n := h^0(X,L) - 1$$ 
and consider the linearly normal embedding
$$
X \hookrightarrow \PP^n
$$
that is determined by the complete linear series
$$
V = \H^0(X,L) \text{.}
$$
Then, there is a bijective correspondence between, on the one hand
\begin{itemize}
\item{
the collection of weight vectors
$$
\mathbf{c} = (c_0,\dots,c_n) \in \ZZ^{n+1} \text{,}
$$
with
$$
0 \leq c_0 \leq \dots \leq c_n \text{;}
$$
and, on the other hand
}
\item{
the set of isomorphism classes of very ample test configurations for $(X,L)$.
}
\end{itemize}
\end{theorem}

The proof of Theorem \ref{weight:vectors:very:ample:test:configs}, is essentially due to Donaldson \cite[Lemma 2]{Donaldson:2005}.  We include a proof, closely following \cite[Proposition 3.7]{Ross:Thomas:2007} and \cite[Section 2.3]{Boucksom-Hisamoto-Jonsson:2016}.

\begin{proof}[Proof of Theorem \ref{weight:vectors:very:ample:test:configs}]
Given a weight vector
$$
\mathbf{c} = (c_0,\dots,c_n) \in \ZZ^{n+1} \text{,}
$$
with 
$$
0 \leq c_0 \leq \dots \leq c_n \text{,}
$$
let
$$
\rho(z) \colon \G_m \rightarrow \GL_{n+1}(\overline{\KK})
$$
be the one parameter subgroup that it determines.  In more explicit terms
$$
\rho(z) := \operatorname{diag}(z^{c_0},\dots,z^{c_n}) \text{.}
$$
Then, the corresponding very ample test configuration is obtained by letting
$$
\mathcal{X} \subseteq \PP^n \times \AA^1 
$$
be the schematic closure of the closed embedding
$$
X \times \G_m \hookrightarrow \PP^n \times \G_m \text{,}
$$
that is defined by
$$
(x,z) \mapsto (\rho(z) \cdot x , z) \text{,}
$$
and 
$$
\mathcal{L} := \Osh_{\mathcal{X}}(1) = \Osh_{\PP^n_{\AA^1}}(1)|_{\mathcal{X}} \text{.}
$$
Conversely, given a very ample test configuration, $(\mathcal{X},\mathcal{L})$, for $(X,L)$, let 
$$
\pi \colon \mathcal{X} \rightarrow \AA^1 
$$
be the projection of $\mathcal{X}$ to $\AA^1$.  Then $\pi_* \mathcal{L}$ is a $\G_m$-linearized rank $n+1$ vector bundle on $\AA^1$.  It is thus $\G_m$-equivariently isomorphic to the trivial bundle $\AA^{n+1} \times \AA^1$ for a suitable one parameter subgroup
$$
\rho(z) := z^{-k} \cdot \operatorname{diag}(z^{c_0},\dots,z^{c_n}) \text{.}
$$
Finally, if two such very ample test configurations determine the same weight vector in this way, then they are isomorphic, as very ample test configurations of $(X,L)$.
\end{proof}

\begin{remark}
As is apparent in the proof of Theorem \ref{weight:vectors:very:ample:test:configs}, see also \cite[Section 2.3]{Boucksom-Hisamoto-Jonsson:2016}, the very ample test configuration $(\mathcal{X},\mathcal{L})$ that corresponds to a weight vector \eqref{1:PS:eqn4} has the property that its special fibre $(\mathcal{X}_{0}, \mathcal{L}|_{\mathcal{X}_{0}})$ is the flat limit, as $z \to 0$, of the flat family $(\mathcal{X}_z,\mathcal{L}|_{\mathcal{X}_z})$, $z \in \AA^1 \setminus \{0\}$.  \end{remark}

\subsection*{Standard $\ZZ$-filtrations} Let $L$ be a semi-ample line bundle on a projective variety $X$.  In what follows, we say that the finitely generated section ring
\begin{equation}\label{ZZ:filt:eqn:1}
R = R(X,L):= \bigoplus_{m \geq 0} R_m
\end{equation}
is \emph{generated in degree $k \in \NN$}, if the $\mathbb{N}$-graded $\overline{\KK}$-algebra
$$
R^{(k)} := \bigoplus_{m \geq 0} R_{km}
$$
is generated, as an $\NN$-graded $\overline{\KK}$-algebra, in degree $1$.  In case that the section ring \eqref{ZZ:filt:eqn:1} is generated, as an $\NN$-graded $\overline{\KK}$-algebra, in degree $1$, then we say that it is \emph{standard}.

Also a $\ZZ$-filtration 
\begin{equation}\label{ZZ:filt:eqn:1:filt:defn}
\mathcal{F}^\bullet = \mathcal{F}^\bullet R \text{,}
\end{equation}
of such a standard section ring \eqref{ZZ:filt:eqn:1}\text{,}
is called \emph{finitely generated} if the bigraded algebra 
\begin{equation}\label{ZZ:filt:eqn:2}
\bigoplus_{(t, m) \in \ZZ \times \ZZ_{\geq 0} } \mathcal{F}^t R_m 
\end{equation}
is finitely generated as a $\overline{\KK}$-algebra.  
As emphasized in \cite{Boucksom-Hisamoto-Jonsson:2016}, a given $\ZZ$-filtration \eqref{ZZ:filt:eqn:1:filt:defn} is finitely generated if and only if the graded $\overline{\KK}[z]$-algebra
\begin{equation}\label{ZZ:filt:eqn3}
\bigoplus_{m \in \ZZ_{\geq 0}} \left( \bigoplus_{t \in \ZZ} z^{- t} \mathcal{F}^t R_m \right)
\end{equation}
is finitely generated.  

Note that such finitely generated $\ZZ$-filtrations are bounded.   Henceforth, we say a finitely generated $\ZZ$-filtration \eqref{ZZ:filt:eqn:2} is a \emph{standard $\ZZ$-filtration}, if it is generated as a $\overline{\KK}[z]$-module in degree $m = 1$.

\subsection*{Projective normality and test configurations}  Consider now, a very ample line bundle $L$ on a projective variety $X$.  Every very ample test configuration $(\mathcal{X},\mathcal{L})$ induces a bounded $\ZZ$-filtration of the graded algebra
\begin{equation}\label{ZZ:filt:eqn4}
R(X,L) := \bigoplus_{m \geq 0} \H^0(X,L^{\otimes m}) \text{.}
\end{equation}
More precisely, as in \cite[Equation (2.1)]{Boucksom-Hisamoto-Jonsson:2016}, a $\ZZ$-filtration on $R(X,L)$ is obtained by letting $\mathcal{F}^{t} \H^0(X,L^{\otimes m})$, for $t \in \ZZ$,  be the image of the $t$-weight part $\H^0(\mathcal{X}, m \mathcal{L})_{t}$ of $\H^0(\mathcal{X}, m \mathcal{L})$ under the injective restriction map
\begin{equation}\label{ZZ:filt:eqn5}
\H^0(\mathcal{X},m\mathcal{L}) \rightarrow \H^0(\mathcal{X}_{1}, m \mathcal{L}|_{\mathcal{X}_{1}}) = \H^0(X,L^{\otimes m}) \text{.}
\end{equation}
As explained in \cite[Proof of Proposition 2.21]{Boucksom-Hisamoto-Jonsson:2016}, such $\ZZ$-filtrations are saturated.  Thus, they are bounded and saturated.

The manner in which very ample test configurations relate to Rees algebras is made precise in the following way.  It makes more precise, for very ample test configurations \cite[Proposition 2.15]{Boucksom-Hisamoto-Jonsson:2016}.

\begin{theorem}[Compare with {\cite[Proposition 2.15]{Boucksom-Hisamoto-Jonsson:2016}}]\label{projective:normal:test:configs:Rees:algebras}
Assume that $L$ is a very ample line bundle on a projective variety $X$ and that $(X,L)$ is projectively normal.    Then, there is a bijective correspondence between very ample test configurations for $(X,L)$ and standard saturated $\ZZ$-filtrations of the section ring $R(X,L)$.
\end{theorem}

\begin{proof}
 Let $n = h^0(X,L) - 1$.  By assumption, the embedding 
$$
X \hookrightarrow \PP^n
$$
that is afforded by the complete linear series $|L|$ is projectively normal.  Thus, the section ring 
$R(X,L)$ is a standard $\overline{\KK}$-algebra.

Let $(\mathcal{X},\mathcal{L})$ be a very ample test configuration for $(X,L)$.  Then, by the description of the corresponding filtration given in \eqref{ZZ:filt:eqn5}, it follows that the corresponding bigraded algebra \eqref{ZZ:filt:eqn3} is generated in degree $m = 1$ as a $\overline{\KK}[z]$-algebra.  Indeed, the reason is that the section ring \eqref{ZZ:filt:eqn4} is generated in degree $m = 1$, as a $\overline{\KK}$-algebra and, moreover, the filtration in degree $m = 1$ is separated and exhaustive.  Finally, the proof of \cite[Proposition 2.21]{Boucksom-Hisamoto-Jonsson:2016} shows, in a more general context, that the filtrations obtained via \eqref{ZZ:filt:eqn5} are indeed saturated.  

Conversely, assume given a finitely generated saturated standard $\ZZ$-filtration of the form \eqref{ZZ:filt:eqn3}.  Then, it is generated, as a $\overline{\KK}[z]$-algebra, in degree $m = 1$.  The very ample test configuration $(\mathcal{X},\mathcal{L})$ is obtained via the relative projective spectrum over $\AA^1$.
\end{proof}

\begin{remark}\label{remark:projective:normal:test:configs:Rees:algebras}  Continuing within the context of Theorem  \ref{projective:normal:test:configs:Rees:algebras}, it is also important to note that although $(X,L)$ is projectively normal, the special fibre $(\mathcal{X}_0,\mathcal{L}|_{\mathcal{X}_0})$ need not be projectively normal in general.  As one example, consider the closed embedding
$$
\PP^1 \times \G_m \hookrightarrow \PP^3 \times \G_m
$$
that is defined by the weight vector
$$
\mathbf{c} = (0,1,3,4) \in \ZZ^4.
$$
\end{remark}

\begin{remark}\label{very:ample:test:config:saturated}
Note that the conclusion of Theorem \ref{vanishing:number:normalized:chow:weights:cor} applies to the class of filtrations that are produced by Theorem \ref{projective:normal:test:configs:Rees:algebras}.
\end{remark}

\section{Filtrations and discrete measures determined by orders of vanishing}\label{real:valuations:discrete:measures}
We now assume that $X$ is a normal $d$-dimensional projective variety.  Let $ \overline{\KK}(X)$ be its function field.  We study filtrations of linear series which are induced by divisors over $X$.  See Theorem \ref{volume:constant:expectation:theorem}.
Our conventions about divisorial (or algebraic) valuations on $X$ are similar to those of \cite[Definition 2.24]{Kollar:Mori:1998}.  

By a slight abuse of perspective, we interchangeably identify prime divisors over $X$ and the divisorial valuations to which they determine.  In particular, if $E$ is a prime divisor over $X$, then $E$ is a nonzero, irreducible, reduced and effective Cartier divisor on some normal proper birational model 
of $X$. We also say that $E$ is a \emph{prime divisor} over $X$.

Let $L$ be a big line bundle on $X$, fix $m \geq 0$ and let $E$ be a prime divisor over $X$ and supported on some normal proper model 
$$\pi \colon X' \rightarrow X\text{.}$$   
There is a well-defined function
\begin{equation}\label{flag:eqn12}
\operatorname{ord}(\cdot) \colon \H^0(X, L^{\otimes m}) \rightarrow [0,\infty] \text{.}
\end{equation}
It is achieved via orders of vanishing along $E$.  It 
 determines a filtration $\mathcal{F}^\bullet_{\operatorname{ord}(\cdot)}$ of $R(L)$.  This filtration has the property that
 \begin{align}\label{divisor:over:number}
\mathcal{F}_{\operatorname{ord}(\cdot)}^t \H^0(X, L^{\otimes m}) & := \{ \sigma \in \H^0(X, L^{\otimes m}) : \operatorname{ord}(\sigma) \geq t \}, \nonumber \\
& = \H^0(X', \pi^* L^{\otimes m} \otimes \Osh_{X'}(- \lceil t E \rceil))  \text{.}
\end{align}
for each $m \geq 0$ and each $t \in \RR$.  

In what follows, let
$$
\beta_E(L) := \int_0^\infty \frac{ \Vol_{X'}(\pi^*L - t E)}{\Vol_X(L) } \mathrm{d}t \text{.}
$$

Let $\mathcal{F}_{\operatorname{ord}(\cdot)}^\bullet$ be the filtration of $R(L)$ induced by $E$.  As in \cite[Lemma 2.22]{Boucksom:Kuronya:Maclean:Szemberg:2015}, for each $t < a_{\max}(||L||)$, put
$$ \operatorname{Vol}(L,v\geq t) := \lim_{m\to \infty} \frac{d!}{m^d} \dim \mathcal{F}_{\operatorname{ord}(\cdot)}^{mt} \H^0(X, L^{\otimes m}).$$

Proposition \ref{Concave:transform:theorem1} below is a consequence of Proposition \ref{expectation:concave:transform}.  We include it here, to place emphasis on the class of filtrations that arise from prime divisors $E$ over $X$.

\begin{proposition}\label{Concave:transform:theorem1}
Let $X$ be a normal projective variety.  Fix  a prime divisor $E$
over $X$ and supported on a normal proper model 
$$\pi \colon X' \rightarrow X\text{.}$$
Let $L$ be a big line bundle on $X$.  
Let $G_{\mathcal{F}^\bullet_{\operatorname{ord}(\cdot)}}$ denote the concave transform determined by $L$ and $\operatorname{ord}(\cdot)$, the divisorial valuation that is determined by $E$.  Here, $\mathcal{F}^\bullet_{\operatorname{ord}(\cdot)}$ is the filtration that is described in \eqref{divisor:over:number}.  Finally, let $\lambda$ be the restriction of Lebesgue measure to the interior of the Okounkov body of $L$.  Then, with these notations and hypothesis, the limit expectation $\mathbb{E}(\nu)$, of the measures $\nu_m$, can be described as
$$ 
\mathbb{E}(\nu) = \int_0^{a_{\max}(||L||)} \frac{\operatorname{Vol}(L,v \geq t)}{\operatorname{Vol}(L)} \mathrm{d}t = \frac{d!}{\operatorname{Vol}(L)} \int_0^{a_{\max}(||L||)} t \cdot \mathrm{d}((G_{\mathcal{F}^\bullet_{\operatorname{ord}(\cdot)}})_* \lambda).
$$
\end{proposition}

\begin{proof}
The measure $\nu$ is
\begin{equation}
\nu = \nu(t) = \lim_{m \to \infty} \nu_m.
\end{equation}
Thus, using the relations \eqref{expectation:concave:transform:eqn} and \eqref{mu:expectation:integral:volume:function}, which we obtained in Proposition \ref{expectation:concave:transform}, it follows that
$$ \mathbb{E}(\nu) = \int_0^{a_{\max}(||L||)} \frac{\operatorname{Vol}(L,v \geq t)}{\operatorname{Vol}(L)} \mathrm{d}t = \frac{d!}{\operatorname{Vol}(L)} \int_0^{a_{\max}(||L||)} t \cdot \mathrm{d}((G_{\mathcal{F}^\bullet_{\operatorname{ord}(\cdot)}})_*\lambda)$$
as desired.
\end{proof}

In Theorem \ref{volume:constant:expectation:theorem}, we give a collection of equivalent descriptions for the expected orders of vanishing constants.  Before stating that result, recall, as in \cite[Definition 4.1]{Codogni:Patakfalvi:2021}, the concept of \emph{$m$-basis type divisor} with respect to the big line bundle $L$ on $X$.   Specifically, a nonzero effective $\QQ$-Cartier divisor $D$ is said to be of \emph{$m$-basis type with respect to $L$}, if there exists Cartier divisors 
$$
D_i \in |L^{\otimes m}| \text{,}
$$
for $i = 0,\dots, n_m - 1$, 
for which their corresponding sections
$$
\sigma_i \in \H^0(X,L^{\otimes m})
$$
are $\overline{\KK}$-linearly independent and for which the given $\QQ$-Cartier divisor $D$ may be expressed as
$$
D = \frac{1}{m h^0(X,L^{\otimes m})} \sum_{i = 0}^{n_m} D_i \text{.}
$$
To indicate the dependence on $L$, we say that such a divisor $D$ is an \emph{$L$-$m$-basis type divisor}.

Theorem \ref{volume:constant:expectation:theorem} is formulated in the following way.

\begin{theorem}\label{volume:constant:expectation:theorem}
Let $X$ be a $d$-dimensional normal projective variety.  Suppose that $L$ is a big 
line bundle on $X$. 
The following assertions hold true.
\begin{enumerate}
\item{
Suppose that $E$ is a prime Cartier divisor over $X$ and supported on a normal proper model 
$$\pi \colon X' \rightarrow X\text{.}$$  
Let $\mathcal{F}^\bullet$ be the induced filtration on the section ring $R(L)$.  Then
the asymptotic volume constant $\beta_E(L)$ equals the limit expectation $\mathbb{E}(\nu)$
\begin{equation}\label{beta:expectation:1}
\beta_E(L) = \mathbb{E}(\nu) = \int_0^{a_{\max}(||L||)} \frac{ t \cdot \operatorname{Vol}_{X'|E}(\pi^*L-tE)}{\operatorname{Vol}(L) / d} \mathrm{d} t  \text{.} 
\end{equation}
Further, 
\begin{equation}\label{beta:expectation:basis:type}
\beta_E(L) = \sup_{\substack{\text{$L$-$m$-basis } \\ \text{ type divisors $D$} } } \operatorname{ord}_E(D) \text{.}
\end{equation}
}
\item{Suppose that $E$ is the exceptional divisor of a blowing-up morphism 
$$\pi \colon X' \rightarrow X$$ 
with center a subscheme $Z \subsetneq X$.  Let $\mathcal{F}^\bullet$ be the induced filtration on $R(L)$.  
Let $G_{\mathcal{F}^\bullet}$ denote the concave transform determined by $L$ and $v$ and denote by $\lambda$ the restriction of the Lebesgue measure to the interior of the Okounkov body $\Delta(L)$.  Then, with these notations and hypothesis, the asymptotic volume constant $\beta_Z(L)$ can be described as
\begin{equation}\label{beta:expectation:2}
\beta_Z(L)  = \mathbb{E}(\nu)  = \frac{d!}{\operatorname{Vol}(L)} \int_0^{a_{\max}(||L||)} t \cdot \mathrm{d}((G_{\mathcal{F}^\bullet_{\operatorname{ord}(\cdot)}})_*\lambda) \text{.}
\end{equation}
Further, when $L$ is assumed to be very ample,
the asymptotic volume constant $\beta_Z(L)$ can be described as a normalized Chow weight
\begin{equation}\label{beta:expectation:3}
\beta_Z(L) = \mathbb{E}(\nu) = \frac{e_X(\mathbf{c})}{(d+1) (\operatorname{deg}_L X)}. 
\end{equation}
Here $e_X(\mathbf{c})$ is the Chow weight of $X$ in $\PP^n_{\overline{\KK}}$ with respect to an \emph{inflectionary embedding}, i.e., with respect to a filtered basis for $\H^0(X,L)$ and 
$$\mathbf{c} = (a_0(L),\dots,a_n(L))$$ is the weight vector determined by the vanishing numbers of $L$ with respect to the filtration $\mathcal{F}^\bullet$.
}
\end{enumerate}
\end{theorem}

\begin{proof}
[Proof of Theorem  \ref{volume:constant:expectation:theorem} and Theorem  \ref{volume:constant:expectation:theorem:intro}]
Within the context of (a), as established in \cite[Theorem 1.2]{Grieve:MVT:2019}, see also  \cite[Theorem 2.24]{Boucksom:Kuronya:Maclean:Szemberg:2015}, the limit measure $\nu$ can be described as
\begin{equation}\label{restricted:volume:measure} 
\nu = \frac{\operatorname{Vol}_{X'|E}(\pi^*L-tE)}{\operatorname{Vol}(L)/d} \mathrm{d} t \text{.}
\end{equation}
Further 
$$
a_{\max}(||L||) = \sup \left\{t > 0 : \pi^*L - t E \text{ is big}\right\}\text{.}
$$
The equation
\eqref{beta:expectation:1} 
follows from Proposition \ref{Concave:transform:theorem1}, the description of the measure $\nu$ given in \eqref{restricted:volume:measure}, together with the fact that, by normality, 
$$\operatorname{Vol}(L,v \geq t) = \operatorname{Vol}(\pi^* L - t E).$$  To complete the proof of (a), it remains to establish \eqref{beta:expectation:basis:type}.  This amounts to establishing the equality
$$
\sup_{\substack{\text{$L$-$m$-basis} \\ \text{
type divisors $D$}} } \operatorname{ord}_E(D) = \frac{1}{m h^0(X,L^{\otimes m})} \sum_{t} t \dim\left( \mathcal{F}^t \H^0(X,L^{\otimes m}) / \mathcal{F}^{t + 1} \H^0(X,L^{\otimes m}) \right) \text{.}
$$
To this end, we argue as in \cite[proof of Lemma 2.2]{Fujita:Odaka:2018}.  Consider an $L$-$m$-basis type divisor
$$
D = \frac{1}{m h^0(X,L^{\otimes m})} \sum_{j = 0}^{n_m} D_j \text{,}
$$
with corresponding basis 
$$
\sigma_0, \dots, \sigma_{n_m} \in \H^0(X,L^{\otimes m}) \text{.}
$$
Then, upon rearranging the order of the $\sigma_j$, if necessary, there exists a decreasing sequence of integers
$$
h^0(X,L^{\otimes m}) \geq b_0 \geq b_1 \geq \dots \geq 0 
$$
such that if
$$
b_t \geq j > b_{t + 1} \text{,}
$$
then
$$
\operatorname{ord}_E(\sigma_j) = t\text{.}
$$
The proof is complete, upon noting that
$$
b_t \leq h^0(X', \pi^* L^{\otimes m} \otimes \Osh_{X'}(- t E))\text{;
}
$$
with equality holding when the basis $\sigma_0,\dots,\sigma_{n_m}$ is compatible with the orders of vanishing filtration $\mathcal{F}^\bullet R_m$.

For (b), Equation \eqref{beta:expectation:2} is implied by Proposition \ref{expectation:concave:transform}.  Finally, \eqref{beta:expectation:3} follows from \eqref{beta:expectation:2} combined with Theorem \ref{vanishing:number:normalized:chow:weights:cor}.  
\end{proof}

\section{Examples}\label{concave:transform:examples}

In this section, we illustrate briefly, our results that are within the context of examples from \cite{Boucksom:Kuronya:Maclean:Szemberg:2015} and \cite{Fujita:2017}.  Additional examples, within the setting of toric varieties, may be found in \cite{Grieve:toric:gcd:2019}.  Finally, we mention that for the case of Fano manifolds, there is another method for calculation of these quantities that uses a minimal model program with scaling.  This approach may be extracted from \cite[Section 3.3]{Fujita:2016a} and \cite[Section 2]{Fujita:2017}.  We give a small discussion about that topic in Example \ref{DF:Example}.  

\begin{example}
If $L$ is an ample line bundle on a nonsingular curve $X$ and if $p \in X$, then the corresponding  Okounkov body $\Delta(L)$ is the interval
$$
\Delta(L) = [0,\deg L] \subseteq \RR.
$$
If $v = \operatorname{ord}_q(\cdot)$, for $q \in X$, then the concave transform 
$$
G_{\mathcal{F}^\bullet_v} \colon [0,\deg L] \rightarrow \RR
$$
is defined by
$$
G_{\mathcal{F}^\bullet_v} =
\begin{cases}
 G_{\mathcal{F}^\bullet_v}(t) = t, & \text{when $q = p$; and} \\
 G_{\mathcal{F}^\bullet_v}(t) = \operatorname{deg} L - t, & \text{ otherwise,}
\end{cases}
$$
see \cite[page 829]{Boucksom:Kuronya:Maclean:Szemberg:2015}.
It follows that
$$\mathbb{E}(\nu) = \beta_q(L) =\frac{(1)!}{\operatorname{Vol}(L)} \int_0^{\operatorname{deg} L} t \cdot \mathrm{d}((  G_{\mathcal{F}^\bullet_v})_*\lambda) = \frac{\operatorname{deg} L}{2} \text{.}
$$
\end{example}

\begin{example}

Suppose that $X = \PP^2_{\overline{\KK}}$ and $L = \Osh_{\PP^2_{\overline{\KK}}}(1)$.  Consider the flag defined by a point $p$ on a line $\ell$.  The Okounkov body $\Delta(L)$ is then the simplex
$$
\Delta(L) = \{ (x,y) \in \RR^2_{+} : x + y \leq 1\}.
$$
Further, let $v = \ord_z(\cdot)$, for a point $z \in X$.  Then
$$
G_{\mathcal{F}^\bullet_v} = 
\begin{cases}
G_{\mathcal{F}^\bullet_v}(x,y) = x + y & \text{when $z = p$; and } \\
G_{\mathcal{F}^\bullet_v}(x,y) = 1 - x \text{,} &  \text{otherwise,}
\end{cases}
$$
see \cite[page 829]{Boucksom:Kuronya:Maclean:Szemberg:2015}.  Thus 
$$
\mathbb{E}(\nu) = 
\frac{(2)!}{\operatorname{Vol}(L)} \int_0^1 t \cdot \mathrm{d}((G_{\mathcal{F}^\bullet_v})_*\lambda) = \frac{(2)!}{\operatorname{Vol}(L)} \iint_{\Delta(L)^\circ} t \circ G_{\mathcal{F}^\bullet_v} \mathrm{d}\lambda = \frac{2}{3}.
$$
\end{example}

\begin{example}
Consider the $2$-tuple embedding
$ \phi \colon \PP^1_{s,t} \rightarrow \PP^2_{x,y,z}$, which is
defined by 
$$ [s,t] \mapsto [s^2 : st : t^2].$$
Put 
$ L = \Osh_{\PP^1}(2) = \phi^* \Osh_{\PP^2}(1) $
and consider the point $p = [1:0] \in \PP^1$.
The sections
$$ \sigma_2 = s^2, \sigma_1 = st, \sigma_0 = t^2 \in \H^0(\PP^1,L)$$
vanish, respectively, to orders
$$ a_2(L) = 2,  a_1(L) = 1 \text{ and } a_0(L) = 0 $$
at $p$.  Put
$ \mathbf{c} = (a_0, a_1, a_2) $
and let $X$ denote the image of $\phi$.  Then $X$ is the degree $2$ plane curve with defining equation given by
$$ xz - y^2 = 0.$$
Thus, as explained in \cite[page 1304]{Evertse:Ferretti:2002} for example, the Chow weight $e_{\PP^1}(\mathbf{c})$ of $\PP^1$ in $\PP^2$ with respect to the embedding $\phi$ and the weights $\mathbf{c}$ is given by
$$e_{\PP^1}(\mathbf{c}) = 2(3) - \min \{ (0,1,2)\cdot (1,0,1), (0,1,2)\cdot(0,2,0) \} = 4. $$
The normalized Chow weight is thus
$$\frac{e_{\PP^1}(\mathbf{c})}{(2)(2)} = \frac{4}{4} = 1;$$
Note also that
$$ \beta_p(L) = \int_0^2 \frac{2 - t}{2} \mathrm{d} t  = 1 = \frac{\operatorname{deg} L}{2}.$$
\end{example}

\begin{example}\label{DF:Example}  
Suppose that $(X,-\K_X)$ is a $d$-dimensional nonsingular Fano variety and that 
$$Z \subsetneq X$$ 
is a codimension $r$ nonsingular subvariety.  Let 
$$\pi \colon X' = \operatorname{Bl}_Z(X) \rightarrow X$$ 
be the blowing-up of $X$ along $Z$, with exceptional divisor $E$.  Recall, that $Z$ is called \emph{dreamy} in case that the bigraded algebra
$$
\bigoplus_{m,t \geq 0} \H^0(X',\pi^*\K_X^{\otimes -m} \otimes \Osh_{X'}(- t E))
$$
is finitely generated.  

For such dreamy nonsingular subvarieties $Z \subsetneq X$, there is a construction of \emph{basic test configuration} $(\mathcal{B},\mathcal{L})$ for $(X,-\K_X)$, with respect to $Z$, and depending on a choice of sufficiently divisible positive integer $m_0 \gg 0$.  This is described in \cite[Section 3.2]{Fujita:2017}.  Further, as noted in \cite[Proposition 3.9 and Remark 3.11]{Fujita:2017}, the Donaldson-Futaki invariant of such a basic test configuration is described as 
\begin{equation}\label{basic:test:config:eqn0}
\operatorname{DF}(\mathcal{B},\mathcal{L}) = \frac{ m_0^{2d} \Vol(-\K_X)^2}{2 (d!)^2} \left(r - \beta_Z(-\K_X) \right) \text{.}
\end{equation}
On the other hand, we know from \eqref{epsilon:beta:eqn} that
\begin{equation}\label{basic:test:config:eqn1}
\beta_Z(-\K_X) \geq \frac{r}{d+1} \epsilon(-\K_X;Z).
\end{equation}
It follows, upon combining \eqref{basic:test:config:eqn0} and \eqref{basic:test:config:eqn1}, having first fixed a sufficiently divisible positive integer $m_0 \gg 0$, that
$$
\operatorname{DF}(X,-\K_X;Z) \leq \frac{r m_0^{2d} \Vol_X(-\K_X)^2}{2(d!)^2} \left( 1 - \frac{1}{d+1} \epsilon(-\K_X;Z) \right) \text{.}
$$

Let us now outline how the quantity \eqref{basic:test:config:eqn0} may be computed following \cite{Fujita:2016a} and \cite{Fujita:2017}, via a Minimal Model Program with scaling.  
Let
$$
\gamma = \gamma(Z) := \sup \{ t \in \RR_{\geq 0} : - \pi^*\K_X - t E \text{ is pseudoeffective} \}
$$ 
be the pseudoeffective threshold and let $\{ (\gamma_i, \pi_i \colon X' \dashrightarrow X) \}_{1 \leq i \leq \ell}$ be the \emph{ample model sequence} \cite[Definition 5]{Fujita:2016a}.  Recall, that the ample model sequence admits an interpretation via an application of a Minimal Model Program with scaling \cite[Section 8]{Fujita:2016a}.  Note that the number of terms which arises in this ample sequence gives a conceptual measure for the extent to which the pseudoeffective threshold $\gamma(Z)$ differs from the Seshadri constant $\epsilon(-\K_X;Z)$.  For example, $\ell = 1$ when $\gamma(Z) = \epsilon(-\K_X;Z)$.

Set $E_i := (\pi_i)_*E$, for $i = 1,\dots,\ell$.  Fix a sufficiently divisible positive integer $m_0 \gg 0$ that has the property that the graded $\overline{\KK}$-algebra
$$
\bigoplus_{\substack{m \geq 0 \\ 0 \leq t \leq m m_0 \gamma}} \H^0(X',\pi^*\K_X^{\otimes -mm_0}  \otimes \Osh_{X'}(- tE))
$$
is generated by 
$$
\bigoplus_{0 \leq t \leq m_0 \gamma} \H^0(X',\pi^*\K_X^{\otimes -m_0}  \otimes \Osh_{X'}(- t E)) \text{.}
$$
Then, as in \cite[Proposition 3.9]{Fujita:2017}, the Donaldson-Futaki invariant \eqref{basic:test:config:eqn0} of the corresponding basic test configuration $(\mathcal{B},\mathcal{L})$ can be described as
$$
\operatorname{DF}(\mathcal{B},\mathcal{L}) = \frac{m_0^{2d}\Vol_X(-\K_X) }{2(d!)^2 } \eta(Z)
$$
where
\begin{align*}
\eta(Z) & =   \sum_{i=1}^\ell \int_{\gamma_{i-1}}^{\gamma_i} \left( d (r-t) \left( -\K_{X_i} + (r-1-t)E_i \right)^{d-1} \cdot E_i \right) \mathrm{d}t  \\
& = r \cdot \Vol_X(-\K_X) - \int_0^\gamma \Vol_{X'}(-\pi^*\K_X - t E) \mathrm{d}t \\
& \leq r \cdot \Vol_X(-\K_X)\left(1 - \frac{\epsilon(-\K_X;Z) }{d + 1} \right) \text{.}
\end{align*}
\end{example}

\begin{example}\label{codim:seshadri:lower:bound}
Here, we expand upon the techniques from \cite[Section 4]{Zhu:2020} and obtain the inequality \eqref{epsilon:beta:eqn}, via different techniques \cite[Theorem 4.2]{Heier:Levin:2017}.   Let $X$ be a normal Cohen-Macaulay projective $\overline{\KK}$-variety and let $L$ be an ample line bundle on $X$.    Put
$
d := \dim X \text{.}
$
Consider a codimension $r$, global complete intersection $Z \subsetneq X$ in the linear system $|L|$.  Denote by $\epsilon(L;Z)$ its \emph{Seshadri constant}.  It is defined by the condition that
$$
\epsilon(L;Z) := \sup \{t \in \RR_{\geq 0} : \pi^*L - t E \text{ is nef} \} \text{.}
$$
Here, $E$ is the exceptional divisor of the blowing-up morphism 
$$\pi \colon X' = \operatorname{Bl}_Z(X) \rightarrow X\text{.}$$    
There is the evident inequality
\begin{equation}\label{seshadri:inequality}
\frac{1}{\Vol(L)} \int_0^{\infty} \Vol(\pi^* L - t E) \mathrm{d}t \geq \frac{1}{\Vol(L)} \int_0^{\epsilon(L;Z)} \Vol(\pi^*L - t E) \mathrm{d} t \text{.}
\end{equation}
Henceforth, we want to obtain an explicit lower bound for the righthand side of the inequality \eqref{seshadri:inequality}.

Since 
$$
\Vol(\pi^* L - t E) = (\pi^* L - t E)^d \text{,}
$$
when 
$$t \in (0,\epsilon(L;Z))\text{,}$$ 
it is important to simplify the $d$-fold intersection product
$$
(\pi^*L - t E)^d = \sum_{i=0}^d (-1)^i \binom{d}{i} t^i \left(\pi^* L^{d-i} \cdot E^i \right)\text{.}
$$
To this end, observe that if 
$$M = \pi^*L - E\text{,}$$ 
then
\begin{align*}
\pi^*L^d - (\pi^* L - t E)^d & = t E \cdot \left(  \sum_{i=0}^{d-1} \pi^* L^i \left((1-t) \pi^* L + t M  \right)^{d-1-i} \right) \\
& = \pi^*L^d \left(\sum_{i=0}^{d-r} \binom{d-1-i}{r-1}(1-t)^{d-r-i} \cdot  t^r \right) \text{.}
\end{align*}

The reason is that $X$ is Cohen-Macaulay and $Z$ is a complete intersection in the complete linear system $|L|$.  Thus 
$$
\pi^* L^{d-i} \cdot (M^{i-1} \cdot E) = \begin{cases}
L^d & \text{ if } i = r \\
0 & \text{ if } i \not = r \text{.}
\end{cases}
$$
Consider now, the \emph{incomplete beta functions}
$$
B_{t_0}(d-r-i+1, r+1) := \int_0^{t_0}(1-t)^{d-r-i} \cdot t^r \mathrm{d}t \text{,}
$$
for $
i = 0,\dots,d-r$ and $t_0 \in \RR_{\geq 0}$.  

Then, since
$$
\frac{(\pi^*L - tE)^d}{L^d} = 1 - \left( \sum_{i=0}^{d-r} \binom{d-1-i}{r-1}(1-t)^{d-r-i} \cdot t^r \right) \text{,}
$$
when $t \in (0,\epsilon(L;Z))$, it follows that
\begin{align*}
\beta_Z(L) & := \frac{\int_0^\infty \Vol(\pi^*L - t E) \mathrm{d}t}{ \Vol(L) } \\
& \geq \frac{\int_0^{\epsilon(L;Z)} \Vol(\pi^*L - t E) \mathrm{d}t}{\Vol(L)} \\
& = \epsilon(L;Z) - \sum_{i=0}^{d-r} \binom{d-1-i}{r-1} B_{\epsilon(L;Z)}(d-r-i+1, r+1) \text{.}
\end{align*}
\end{example}

\section{Approximation constants for divisorial valuations}\label{approx:constants}

The purpose of this section, is to define and study approximation constants from the viewpoint of local Weil functions.  The main result is Theorem \ref{Roth:Constants:Thm}.  It gives a logarithmic form of \cite[Theorem 5.1]{McKinnon-Roth} and supplements \cite{Ru:Wang:2016} and \cite{Ru:Wang:2021}.  As in \cite{Grieve:Function:Fields}, the Subspace Theorem of Schmidt is used in place of the theorem of Faltings and W\"{u}stholz \cite{Faltings:Wustholz}.  In Section \ref{Roth:constants:delta:invariant}, we explain its relation to the $\mathrm{K}$-stability $\delta$-invariant.  (See Corollary \ref{K:stab:delta:invariant:cor}.)

Fix a base number field $\KK$.  Let $M_\KK$ denote its set of places.  Let $\overline{\KK}$ be a fixed algebraic closure of $\KK$.  In what follows, our conventions about absolute values are consistent with those of \cite[Section 1.4]{Bombieri:Gubler}.  In particular, they are normalized so that the product formula holds true with absolute value equal to one.  We also employ the theory of local Weil and height functions.  

Again, our conventions extend those of \cite[Sections 2.2 and 2.3]{Bombieri:Gubler}.  They are consistent with those of \cite{Grieve:Divisorial:Instab:Vojta} and \cite{Grieve:points:bounded:degree}.  For example, our conventions for local Weil functions defined by Cartiers divisors, allows for fields of definition which are finite extensions of the given base field.  They are normalized with respect to the base number field $\KK$.   We refer to \cite[Section 3]{Grieve:Divisorial:Instab:Vojta} or \cite[Section 2]{Grieve:points:bounded:degree}, for example, for further details and omit further discussion as to their definitions here.

Let $L$ be an ample line bundle on a geometrically irreducible and geometrically normal projective variety $X$.  We assume that $(X,L)$ is defined over $\KK$.  In what follows, $h_L(\cdot)$ denotes its logarithmic height function.  Suppose that 
$
\operatorname{ord}_E(\cdot) 
$
is a divisorial valuation on $X$ and having field of definition $\FF / \KK$ a finite extension of $\KK$ and having the property that $\KK \subseteq \FF \subseteq \overline{\KK}$.  

Then $\operatorname{ord}_E(\cdot)$ is the valuation on  $\FF(X_{\FF})$ that is obtained via orders of vanishing along $E$ a geometrically irreducible, reduced and effective Cartier divisor on some normal proper model of $X_{\FF}$ and defined over $\FF$.  
Fixing such a model 
$$\pi \colon X' \rightarrow X_{\FF}\text{,}$$ 
with $E \subseteq X'$, the corresponding filtration $\mathcal{F}^\bullet R$, of the section ring $R$, is described as
\begin{align*}
\mathcal{F}^t R_m & = \H^0\left(X_{\FF}, L_{\FF}^{\otimes m} \otimes \pi_* (\Osh_{X'}(- \lceil t E \rceil ))\right) \\
&= \H^0\left(X', \pi^* L^{\otimes m}_{\FF} \otimes \Osh_{X'}(- \lceil t E \rceil)\right)
\end{align*}
for all 
$m \geq 0$ 
and all 
$t \in \ZZ \text{.}$  

Inside of 
$$X_{\overline{\KK}} = X \times_{\spec \KK} \spec \overline{\KK}\text{,}$$ 
the Zariski closure of the center of $E$ will be denoted by $\operatorname{center}_X(E)$.  Moreover, given a place $v \in M_{\KK}$, let $\lambda_v(\cdot,E)$ be the corresponding local Weil function.  It is normalized with respect to the base number field $\KK$.  Its domain is the set of those $\FF$-rational points $x \in X(\FF)$ which are not contained in $\operatorname{center}_X(E)$.   We adopt similar notation for  local Weil functions determined by global sections of line bundles.

In Definition \ref{divisorial:val:approx:contant}, we give a logarithmic formulation of approximation constants for divisorial valuations.  It should be compared  with {\cite[Definition 2.8, 2.9]{McKinnon-Roth}} and {\cite[Section 3.5]{Grieve:Function:Fields}}.

\begin{defn}\label{divisorial:val:approx:contant}
Let $L$ be an ample line bundle on a geometrically irreducible and geometrically normal projective variety $X$.  Let 
$\operatorname{ord}_E(\cdot)$  
be a divisorial valuation on $X$ determined by a geometrically irreducible and reduced effective Cartier divisor $E$, over $X$, and having field of definition $\FF / \KK$ a finite extension of $\KK$ and having the property that $\KK \subseteq \FF \subseteq \overline{\KK}$.  Let $v \in M_{\KK}$ be a fixed place of $\KK$ and extended to $\overline{\KK}$.
Let 
$$\{x_i\} \subseteq X(\KK) \setminus \operatorname{center}_X(E)(\KK)$$ 
be an infinite sequence of $\KK$-rational points.  We distinguish amongst two cases.  In the first case, 
$$\lambda_v(x_i,E) \to \infty$$ 
as $i \to \infty$.  In that case, define $\alpha_E(\{x_i\},L)$ to be the infimum of those real numbers 
$$\gamma \in \RR\text{,}$$ which have the property that
$$h_L(x_i) \leq \gamma \lambda_v(x_i,E) + \mathrm{O}(1) \text{,}$$ 
as $i \to \infty$.

In the second case, 
$$\lambda_v(x_i,E) \not \to \infty $$ as $i \to \infty$.  When this happens, we put 
$$\alpha_E(\{x_i\},L) = \infty \text{.}$$
\end{defn}

\begin{remark}
Recall that $\lambda_v(\cdot,E)$ gives a measure of the negative of the $v$-adic logarithmic distance to $E$.  Indeed, this is the viewpoint of Silverman \cite[Section 2]{Silverman:1987}.  Thus, $\lambda_v(x_i,E)$ is large when points are $v$-adically close to $\operatorname{center}_X(E)$.  This explains the intuition behind the property that an infinite sequence of $\KK$-rational points
$$
\{x_i\} \subseteq X(\KK) \setminus \operatorname{center}_X(E)(\KK)
$$
satisfies the condition that
$$
\lambda_v(x_i,E) \to \infty \text{,}
$$
as $i \to \infty$.
\end{remark}

In our present context, Theorem \ref{Roth:Constants:Thm} is our analogue of \cite[Theorem 5.1]{McKinnon-Roth}, which is due to McKinnon and Roth.  The proof is similar, although, as in \cite{Grieve:Function:Fields}, the Subspace Theorem of Schmidt can be used in place of the approximation theorem of Faltings and W\"{u}stholz \cite{Faltings:Wustholz}.  Theorem \ref{Roth:Constants:Thm}, below, complements the work of Ru and Wang (see \cite{Ru:Wang:2016} and \cite{Ru:Wang:2021}) and Heier and Levin \cite{Heier:Levin:2017}.

\begin{theorem}\label{Roth:Constants:Thm}  Let $L$ be an ample line bundle on a geometrically irreducible and geometrically normal projective variety $X$ and defined over a number field $\KK$.  Let 
$S \subseteq M_{\KK}$ 
be a finite subset.  Fix a collection of positive real numbers $\{R_v\}_{v \in S}$ and for each $v \in S$, let $E_v$ be a prime divisor over $X$ and having field of definition some finite extension of $\KK$.

If 
$$
\sum_{v \in S} \beta_{E_v}(L) R_v > 1 \text{,}
$$
then there exists a proper Zariski closed subset 
$$W \subsetneq X$$ 
so that the inequalities
$$
\alpha_{E_v}(\{x_i\},L) \geq 1 / R_v
$$
are valid for all infinite sequences of distinct $\KK$-rational points 
$$
\{x_i\} \subseteq X(\KK) \setminus W(\KK)
$$
and at least one place $v \in S$.
\end{theorem}

\begin{proof}
We combine the approach of \cite[Proof of Theorem 5.1]{McKinnon-Roth} and \cite[Proposition 6.2 and Theorem 6.3]{Grieve:Function:Fields}.  In particular, we apply Schmidt's Subspace Theorem, for linear systems, to construct \emph{vanishing sequences}, which are \emph{Diophantine constraints}.  Our overall outline of argument follows \cite[Proof of Theorem 5.1]{McKinnon-Roth} closely.  First of all, as in \cite[Proof of Theorem 5.1]{McKinnon-Roth}, we may and do assume that the set $X(\KK)$ is Zariski dense.

In what follows, to simplify notation, we refrain from explicit mention of base change to appropriate fields of definition.  For example, if $\FF / \KK$ is a finite extension field, with $\KK \subseteq \FF \subseteq \overline{\KK}$, then we simply write $X$ and $L$ in place of  $X_{\FF}$ and $L_{\FF}$, the respective base change of $X$ and $L$ with respect to the base change morphism $\operatorname{Spec} \FF \rightarrow \operatorname{Spec} \KK$.

By assumption
$$
\sum_{v \in S} \beta_{E_v}(L) R_v > 1 \text{.}
$$
For each $v \in S$, let $\mathcal{F}^\bullet_{E_v}$ be the filtration of the section ring $R(L)$ that is induced by $E_v$.  For each $v \in S$, fix a normal projective model
$$
\pi_v \colon X^{(v)} \rightarrow X
$$
which has the property that 
$$
E_v \subset X^{(v)}
$$
is a Cartier divisor.  Moreover, set
$$
a_{\max}(||L||,v) := \sup \{t > 0 : \pi^*_v L - tE_v \text{ is big} \} \text{.}
$$

Then, as in \cite[Lemma 5.5]{McKinnon-Roth}, for each $v \in S$, there exists rational numbers
$$
0 < t_{v,1} < t_{v,2} < \dots < t_{v,q_v} < a_{\max}(||L||,v)
$$
which have the property that if
$$
c_{v,j} = R_v t_{v,j}
$$
for $v \in S$ and $j = 1,\dots,q_v$, then
\begin{equation}\label{key:eqn:0}
\sum_{v \in S} \left( \sum_{j=1}^{q_v} c_{v,j} \left( g_v(t_{v,j}) - g_v(t_{v,j+1}) \right) \right) > 1 \text{.}
\end{equation}
Here, we have put
$$
g_v(t) = \frac{\Vol(\pi^* L - t E_v)}{\Vol(L)} \text{,}
$$
for $t \geq 0$ and $v \in S$.

Observe now that if $t \geq 0$, then
$$
\frac{\dim \mathcal{F}^t_{E_v} \H^0(X,L^{\otimes m})}{\Vol(L)} \to g_v(t)
$$
as $m \to \infty$.  Thus, in light of the inequality \eqref{key:eqn:0}, we may choose a sufficiently large sufficiently divisible integer $m$ so that the inequality
\begin{equation}\label{key:eqn}
1 < \sum_{v \in S} \frac{1}{h^0(X,L^{\otimes m})}\sum_{j=1}^{q_v} c_{v,j} \left(\dim \left(\mathcal{F}_{E_v}^{mt_{v,j}} \H^0(X,L^{\otimes m}) / \mathcal{F}_{E_v}^{mt_{v,j+1}} \H^0(X,L^{\otimes m}) \right) \right)
\end{equation}
is valid and so that
$$
m t_{v,j} \in \ZZ \text{,}
$$
for all $v \in S$ and all $j = 1,\dots,q_v$.

Recall that 
$$
\mathcal{F}_{E_v}^{m t_{v,j}} \H^0(X,L^{\otimes m}) = \H^0(X^{(v)}, \pi_v^* L^{\otimes m} \otimes \Osh_{X^{(v)}}(- m t_{v,j} E_v)) \text{,}
$$
for all $v \in S$ and all $j = 1,\dots,q_v$.  Further, note that
$$
\mathcal{F}_{E_v}^{m t_{v,j+1} } \H^0(X,L^{\otimes m}) \subseteq \mathcal{F}_{E_v}^{m t_{v,j} } \H^0(X,L^{\otimes m}) \text{,}
$$
for $j = 1,\dots, q_v - 1$.

For each $v \in S$, fix a basis $\{s_{v,j,\ell}\}_{\ell \in I_{v,j}}$ for $\mathcal{F}_{E_v}^{m t_{v,j}} \H^0(X,L^{\otimes m})$ which is compatible with the filtration.  Then, for each fixed $v \in S$, the collection of such sections may be extended to a basis $s_{v,0},\dots,s_{v,n_m}$ for $\H^0(X,L^{\otimes m})$ that is compatible with the filtration $\mathcal{F}_{E_v}^\bullet \H^0(X,L^{\otimes m})$.  

We now note that because of inequality \eqref{key:eqn}, Schmidt's Subspace Theorem (for linear systems), see for example \cite[Theorem 2.6]{Ru:Vojta:2016}, \cite[Theorem 3.3]{Grieve:points:bounded:degree} or  \cite[Proposition 2.1]{Grieve:2018:autissier}, implies that the solutions
\begin{equation}\label{soltns:eqn1}
x \in X(\KK) \setminus \bigcup_{\substack{ v \in S \\ j = 0,\dots, n_m}} \operatorname{Supp}(s_{v,j})(\KK)
\end{equation}
to the simultaneous system of Diophantine inequalities
\begin{equation}\label{soltns:systems:eqn2}
\frac{1}{R_v t_{v,j}} \lambda_{v}(x,s_{v,j,\ell}) \geq h_{L^{\otimes m}}(x) + \mathrm{O}(1)\text{,}
\end{equation}
for each $v \in S$, each $j = 1,\dots, q_v$ and each $\ell \in I_{v,j}$ are contained in some proper Zariski closed subset 
$$Z \subsetneq X\text{.}$$

Finally, suppose that the conclusion of Theorem \ref{Roth:Constants:Thm} is false for this choice of $Z$.  Then, there exists an infinite sequence of rational points
\begin{equation}\label{collection:rational:points:eqn}
\{x_i\} \subseteq X(\KK) \setminus Z(\KK)
\end{equation}
which have the property that
$$
\alpha_{E_v}(\{x_i\},L) < \frac{1}{R_v} \text{.}
$$
In particular, there exists an infinite sequence of rational points \eqref{collection:rational:points:eqn}, which has the property that for all sufficiently small $\delta' > 0$ and all $v \in S$ is holds true that 
$$
\left(\frac{1}{R_v} - \delta'  \right) \lambda_{v}(x_i,E_v) - h_L(x_i) \to \infty
$$
as $i \to \infty$.

Now, each 
$$s_{v,j,\ell} \in \mathcal{F}_{E_v}^{m t_{v,j}} \H^0(X,L^{\otimes m})$$ 
vanishes to order $m t_{v,j}$ along $E_v$.  Thus, for all $\delta > 0$ and all sufficiently large $i$, depending on $\delta$, it follows that 
$$
\lambda_{v}(x_i,s_{v,j,\ell}) \geq (m t_{v,j} - \delta) \lambda_{v}(x_i,E_v) \text{.}
$$
But then it also follows that for all sufficiently large $i$, if $v \in S$, $j = 1,\dots,q_v$ and $\ell \in I_{j,v}$ then, the inequality
$$
\frac{1}{m R_v t_{v,j}} \lambda_{v}(x_i,s_{v,j,\ell}) - h_L(x_i) \geq \left(\frac{1}{R_v} - \frac{\delta}{m R_v t_{v,j}} \right) \lambda_{v}(x_i,E_v) - h_L(x_i) 
$$
is valid.

For small enough $\delta > 0$, the righthand side above tends to $\infty$ as $i \to \infty$.  But this is not compatible with the fact that 
$$Z \subsetneq X$$ 
contains all solutions \eqref{soltns:eqn1} to the simultaneous Diophantine system \eqref{soltns:systems:eqn2} above.  Indeed, it follows from the above discussion that some $x_i$ must lie in $Z$.  This is a contradiction.
\end{proof}

\section{Roth constants and the $\K$-stability $\delta$-invariant}\label{Roth:constants:delta:invariant}

In this final section, we mention how Theorems \ref{Roth:Constants:Thm} and \ref{volume:constant:expectation:theorem} intersect with the $\K$-stability $\delta$-invariant.   This is the content of Corollary \ref{K:stab:delta:invariant:cor}.  It is in the direction of Vojta's Main Conjecture (compare with \cite[Theorem 10.1]{McKinnon-Roth} and with the main results from \cite{Grieve:Divisorial:Instab:Vojta}, \cite{Grieve:2018:autissier}, \cite{Ru:Vojta:2016} and \cite{McKinnon:Santriano:2021}). 

In order to formulate Corollary \ref{K:stab:delta:invariant:cor}, we briefly recall the $\K$-stability $\delta$-invariant.  More details can be found, for instance, in \cite[Section 4]{Codogni:Patakfalvi:2021} and \cite[Theorem C]{Blum:Jonnson:2017}.  Assume that $(X,\Delta)$ is a Kawamata log terminal pair, that is defined over $\KK$.  In particular, by our conventions, $X$ is a geometrically irreducible and geometrically normal projective variety and $\Delta$ is an effective $\QQ$-divisor which has the two properties that
$$\lfloor \Delta \rfloor = 0\text{;}$$ and
$$
a(E,X,\Delta) > - 1 \text{,}
$$
for all divisorial valuations 
$ \operatorname{ord}_E(\cdot)$ determined by nonzero prime Cartier divisors $E$ 
over $X_{\overline{\KK}}$ and having field of definition some finite extension field $\FF / \KK$, with $\KK \subseteq \FF \subseteq \overline{\KK}$.

Here, $a(E,X,\Delta)$ is the \emph{discrepancy} of $(X,\Delta)$ with respect to $E$.  
We refer to the text of Koll\'{a}r and Mori \cite[Section 2.3]{Kollar:Mori:1998} for more details about Kawamata log terminal pairs.

Now, fixing an ample line bundle $L$ on $X$, the $\K$-stability $\delta$-invariant is described by the condition that
$$
\delta(X,\Delta;L) := \inf_{\substack{\text{nontrivial divisorial valuations $\operatorname{ord}_E(\cdot)$} \\ \text{ over $X_{\overline{\KK}}$}}} \frac{A(E,X,\Delta)}{ \beta_E(L)} \text{.}
$$
Here
$$
A(E,X,\Delta) := a(E,X,\Delta) + 1
$$
is the \emph{log discrepancy} of $(X,\Delta)$, with respect to $E$, the prime Cartier divisor that corresponds to the divisorial valuation $\operatorname{ord}_E(\cdot)$.

Recall, the special case that $(X,\Delta)$ is a \emph{Fano pair}, in the sense that  the anti-log-canonical divisor $-\K_X - \Delta$ is ample.  Then, $(X,\Delta)$ is \emph{$\K$-semistable} if and only if 
$$\delta(X,\Delta;-\K_X - \Delta) \geq 1\text{.}$$  
On the other hand, $(X,\Delta)$ is \emph{uniformly $\K$-stable} if and only if $\delta(X,\Delta) > 1$.  We refer to \cite[Corollary 4.9]{Codogni:Patakfalvi:2021} for more details.

Returning to the general case of a Kawamata log terminal pair $(X,\Delta)$, as above, and $L$ an ample line bundle on $X$, here we consider a form of \emph{arithmetic uniform $\K$-stability}.  In more precise terms, fix a finite set of places $S$ and for each $v \in S$, let $E_v$ be a divisorial valuation over $X$, with field of definition some finite extension field $\FF / \KK$ with the property that $\KK \subseteq \FF \subseteq \overline{\KK}$.  

Then, here, we say that $(X,\Delta)$ is \emph{not arithmetically $\K$-stable} with respect to $L$ and $E_v$, for $v \in S$, if
$$
1 < \sum_{v \in S} A(E_v,X,\Delta) < \sum_{v \in S} \beta_{E_v}(L)R_v
$$
for some positive constants $R_v$.  We say that such positive constants $R_v$, for $v \in S$, are \emph{arithmetically $\K$-destabilizing Roth constants}.

Especially, such considerations motivate our formulation of Corollary \ref{K:stab:delta:invariant:cor} below.

\begin{corollary}\label{K:stab:delta:invariant:cor}
Fix a finite set of places $S \subseteq M_{\KK}$ of the base number field $\KK$.  Let $(X,\Delta)$ be a Kawamata log terminal pair as above.  Let $L$ be an ample line bundle on $X$.  
Fix a collection of divisorial valuations $E_v$, for $v \in S$, and having field of definition defined over some finite extension field of $\KK$.  

If $(X,\Delta)$ is not arithmetically $\K$-stable with respect to $L$ and $E_v$, for $v \in S$, and if $R_v$, for $v \in S$, are arithmetically destabilizing Roth constants, then
there exists a proper Zariski closed subset 
$$W \subsetneq X$$ 
so that the inequalities
$$
\alpha_{E_v}(\{x_i\},L) \geq  1/R_v
$$
are valid for all infinite sequences of distinct $\KK$-rational points 
$$
\{x_i\} \subseteq X(\KK) \setminus W(\KK)
$$
and at least one place $v \in S$.
\end{corollary}

\begin{proof}[Proof of Theorem \ref{Roth:constant:theorem} and Corollary \ref{K:stab:delta:invariant:cor}]
Recall, that the arithmetically $\K$-destabilizing Roth constants $R_v$, for $v \in S$,  satisfy the inequality that
$$
1 < \sum_{v \in S} A(E_v,X,\Delta) < \sum_{v \in S} \beta_{E_v}(L)R_v \text{.}
$$
The conclusion of Theorem \ref{Roth:constant:theorem} and Corollary \ref{K:stab:delta:invariant:cor} thus follows from Theorem \ref{Roth:Constants:Thm}.
\end{proof}

\end{document}